\documentclass[11pt]{amsart}

\usepackage{enumerate}
\usepackage{amsmath,amsthm,verbatim,amssymb,amsfonts,amscd,amsopn,amsxtra,graphicx,lmodern,enumitem}
\usepackage[hidelinks]{hyperref}
\usepackage{tikz-cd} 
\usepackage{graphics}
\graphicspath{ {./images/} }
\usepackage{mathrsfs}
\usepackage{relsize}
\usepackage{enumitem}
\usepackage[utf8]{inputenc}
\usepackage{csquotes}
\usepackage{amsthm}
\usepackage{thmtools}
\usepackage{mathtools}
\usepackage{microtype}
\usepackage{pdfrender,xcolor}

\tikzset{
	symbol/.style={
		draw=none,
		every to/.append style={
			edge node={node [sloped, allow upside down, auto=false]{$#1$}}}
	}
}

\newlist{condenum}{enumerate}{1} 
\setlist[condenum]{label=\bfseries C\arabic*., 
	ref=\arabic*, wide}

\begin{document}
	\pdfrender{StrokeColor=black,TextRenderingMode=2,LineWidth=0.2pt}	
	
	\title{Relative approximation degrees and the Henselian Rationality problem over perfect fields}
	
		\author{Arpan Dutta and Rumi Ghosh}
		\address{Department of Mathematics, School of Basic Sciences, IIT Bhubaneswar, Argul,
			Odisha, India, 752051.}
		\email{arpandutta@iitbbs.ac.in, s23ma09008@iitbbs.ac.in}


	\def\NZQ{\mathbb}               
	\def\NN{{\NZQ N}}
	\def\QQ{{\NZQ Q}}
	\def\ZZ{{\NZQ Z}}
	\def\RR{{\NZQ R}}
	\def\CC{{\NZQ C}}
	\def\AA{{\NZQ A}}
	\def\BB{{\NZQ B}}
	\def\PP{{\NZQ P}}
	\def\FF{{\NZQ F}}
	\def\GG{{\NZQ G}}
	\def\HH{{\NZQ H}}
	\def\UU{{\NZQ U}}
	\def\P{\mathcal P}
	\def\Z{\mathcal Z}
	\def\C{\mathcal C}

	%
	%
	\let\union=\cup
	\let\sect=\cap
	\let\dirsum=\oplus
	\let\tensor=\otimes
	\let\iso=\cong
	\let\Union=\bigcup
	\let\Sect=\bigcap
	\let\Dirsum=\bigoplus
	\let\Tensor=\bigotimes
	
	\theoremstyle{plain}
	\newtheorem{Theorem}{Theorem}[section]
	\newtheorem{Lemma}[Theorem]{Lemma}
	\newtheorem{Corollary}[Theorem]{Corollary}
	\newtheorem{Proposition}[Theorem]{Proposition}
	\newtheorem{Problem}[Theorem]{Problem}
	\newtheorem{Conjecture}[Theorem]{Conjecture}
	\newtheorem{Question}[Theorem]{Question}

	\theoremstyle{definition}
	\newtheorem{Example}[Theorem]{Example}
	\newtheorem{Examples}[Theorem]{Examples}
	\newtheorem{Definition}[Theorem]{Definition}
	
	\theoremstyle{remark}
	\newtheorem{Remark}[Theorem]{Remark}
	\newtheorem{Remarks}[Theorem]{Remarks}

	\newcommand{\n}{\par\noindent}
	\newcommand{\nn}{\par\vskip2pt\noindent}
	\newcommand{\sn}{\par\smallskip\noindent}
	\newcommand{\mn}{\par\medskip\noindent}
	\newcommand{\bn}{\par\bigskip\noindent}
	\newcommand{\pars}{\par\smallskip}
	\newcommand{\parm}{\par\medskip}
	\newcommand{\parb}{\par\bigskip}

	\let\epsilon=\varepsilon
	\let\phi=\varphi
	\let\kappa=\varkappa
	
	\newcommand{\trdeg}{\mbox{\rm trdeg}\,}
	\newcommand{\rr}{\mbox{\rm rat rk}\,}
	\newcommand{\sep}{\mathrm{sep}}
	\newcommand{\ac}{\mathrm{ac}}
	\newcommand{\ins}{\mathrm{ins}}
	\newcommand{\res}{\mathrm{res}}
	\newcommand{\Gal}{\mathrm{Gal}\,}
	\newcommand{\ch}{\operatorname{char}}
	\newcommand{\Aut}{\mathrm{Aut}\,}
	\newcommand{\kras}{\mathrm{kras}\,}
	\newcommand{\dist}{\mathrm{dist}\,}
	\newcommand{\ord}{\mathrm{ord}\,}
	\newcommand{\Div}{\mathrm{Div}\,}
	\newcommand{\Supp}{\mathrm{Supp}\,}
	\newcommand{\Spec}{\mathrm{Spec}\,}
	\newcommand{\height}{\mathrm{ht}\,}
	\newcommand{\rk}{\mathrm{rank}\,}
	\newcommand{\Diff}{\mathrm{Diff}\,}
	\newcommand{\Ram}{\mathrm{Ram}\,}
	\newcommand{\id}{\mathrm{id}\,}
	\newcommand{\lex}{\mathrm{lex}\,}
	\newcommand{\gr}{\mathrm{gr}\,}
	\newcommand{\init}{\mathrm{in}\,}
	\newcommand{\depth}{\mathrm{depth}\,}
	\newcommand{\Lim}{\mathrm{Lim}\,}

	\let\phi=\varphi
	\let\kappa=\varkappa
	
	\def \a {\alpha}
	\def \b {\beta}
	\def \s {\sigma}
	\def \d {\delta}
	\def \g {\gamma}
	\def \o {\omega}
	\def \l {\lambda}
	\def \th {\theta}
	\def \e {$\epsilon$}
	\def \D {\Delta}
	\def \G {\Gamma}
	\def \O {\Omega}
	\def \L {\Lambda}


	%
	%
	\textwidth=15cm \textheight=22cm \topmargin=0.5cm
	\oddsidemargin=0.5cm \evensidemargin=0.5cm \pagestyle{plain}


	
	\date{\today}
	
	\maketitle
	
	
	\begin{abstract}
		Let \((F|K,w)\) be an immediate extension of function fields of transcendence degree one. The extension is said to be henselian rational if there exists some element \(X\) in the henselization \(F^h\) of \((F,w)\) such that \(F^h = K(X)^h\). The problem of henselian rationality over tame fields was settled by Kuhlmann in 2019. In the present paper, we develop two ingredients towards extending the result to perfect fields. 
		
		We first complete the theory of relative approximation degrees in the hitherto unresolved algebraic-type setting needed for this purpose. Over henselian fields, we prove the existence of the relative approximation degree and relative approximation constant of every polynomial. We interpret these invariants through the \(j\)-invariants of suitable associated monomial valuations, thereby showing that they can directly be read off from the corresponding Taylor expansions themselves. As a consequence, we extend the henselian degree bound of Kuhlmann-Vlahu to this setting.
		
		We then study the Artin-Schreier reduction underlying Kuhlmann's proof. Over henselian perfect fields, we show that every polynomial is Artin-Schreier equivalent to one whose relative approximation degree belongs to \(\{1,p\}\). We construct an explicit rank-one example illustrating that this bound is sharp. Therefore, Kuhlmann's polynomial reduction method may not be carried over directly to the setup of perfect fields. Nevertheless, our example does not disprove henselian rationality. Indeed, the associated Artin-Schreier function field is shown to be henselian rational. 
		
		Finally, under the assumption that \((K,v)\) equals its absolute ramification field, we prove that \((F^h|L,w)\) is henselian rational, where \(L\) is the implicit constant field of the extension \((F|K,w)\). Moreover, henselian rationality descends from \(L\) to \(K\) whenever \(L|K\) is a finite extension. In particular, this finiteness condition holds whenever some separating transcendental element induces an extension of Type II, yielding henselian rationality in this case. 
	\end{abstract}


\section{Introduction} 
\subsection{Henselian rationality and relative approximation degrees} Let \((K,v)\) be a valued field with \(\ch K = p > 0\) an odd prime. A central problem in the structure theory of valued function fields related closely to the Local Uniformization problem is to determine when an immediate function field of transcendence degree one becomes rational after henselization. Let \( (F|K,w) \) be an \textit{immediate} extension of valued function fields of transcendence degree one, that is, the value group \(wF\) and residue field \(Fw\) coincide with \(vK\) and \(Kv\) respectively. Fix an extension of \(w\) to \( \overline{F} \), a fixed algebraic closure of \(F\). 
\begin{Definition}\label{Defn henselian rationality}
	The extension \( (F|K,w) \) is said to be \textbf{henselian rational} if there exists an element \(Y\) in the \textit{henselization} \(F^h\) of \(F\), transcendental over \(K\), such that \( F^h = K(Y)^h \).
\end{Definition} Over \textit{tame} fields (that is, valued fields whose algebraic closure coincides with their absolute ramification field), Kuhlmann proved henselian rationality in \cite{Kuhlmann2019EliminationII}. His approach can be summarised as follows:
\begin{enumerate}[label=(\Alph*)]
	\item Using ramification theory, reduce to the case when \( \rk(K,v) = 1 \)\label{Step A}.
	\item Prove henselian rationality over algebraically closed base fields of rank one\label{Step B}. 
	\item Prove henselian rationality over tame base fields by \enquote{pulling down} henselian rationality through tame extensions over algebraically maximal fields.\label{Step C} 
\end{enumerate}  
The problem remains open once we remove the tameness hypothesis. Over base fields of positive characteristic, tame fields are precisely those which are \textit{defectless} and perfect. An example has been constructed in \cite{Kuhlmann2001SymbLogic} exhibiting the failure of henselian rationality over a defectless field which is not perfect. Therefore, there is still hope of extending henselian rationality to the setup of perfect fields. A positive resolution of this problem would be a key to an alternate valuation theoretic proof of Temkin's \enquote{Inseparable Local Uniformization} \cite{TemkinInsepLU}.

\pars The primary approach for tackling this problem is the following: after reducing to the rank one setup, pick a separating transcendental element \(X \in F\setminus K\). Then \( (F|K(X),w) \) is a finite separable immediate extension. After lifting through suitable extensions, we can further assume that this extension can be decomposed as a tower of Artin-Schreier extensions, say 
\[  K(X)^h \longrightarrow K(X)^h (\th_1) \longrightarrow \dotsb \longrightarrow K(X)^h(\th_1, \dotsc , \th_n) = F^h,  \]
where each \( \th_i\) is an Artin-Schreier generator over the preceding field. We first focus on \( \th_1 \). Since \( (K(X),w) \) is a rank one valued field, the polynomial ring \(K[X]\) is dense in the henselization \(K(X)^h\), and hence we can assume 
\[ \th_1^p - \th_1 = f(X) \in K[X].  \]
Our goal is to find some polynomial \( F(X) \in K[X] \) such that \( F \equiv f \bmod \P(K[X]) \) where \( \P \) is the Artin-Schreier operator, such that
\[ K(X)^h = K(F(X))^h.  \]
The existence of such a polynomial would allow us to reduce the length of the above Artin-Schreier chain in the following way: let \( \a_1 \) be a root of the polynomial \( T^p - T - F(X) \) over \( K(X)^h \). Since \( f \equiv F \bmod \P(K[X]) \), \( \th_1 \) and \( \a_1 \) generate the same extension over \( K(X)^h \). Consequently, 
\[ K(X)^h(\th_1) = K(X)^h (\a_1) = K(F(X))^h(\a_1) = K(\a_1)^h.  \]
Thus it is important to be able to compute the degrees \( [K(X)^h:K(f(X))^h] \), and especially give conditions under which the degree becomes one. In general, this problem is quite difficult. The relative approximation degree and relative approximation constant were introduced by Kuhlmann and Vlahu in \cite{KuhlmannVlahu2014} to tackle this very issue. Their significance in this context is captured by Theorem 9.1 of \cite{KuhlmannVlahu2014}, which states that
\[ [K(X)^h : K(f(X))^h] \le h_K(X:f),  \]
where \(h_K(X:f)\) denotes the relative approximation degree of \(f\). Therefore, 
\[ K(X)^h = K(f(X))^h \text{ whenever } h_K(X:f) = 1.  \]
Thus the problem of henselian rationality reduces to the problem of finding an Artin-Schreier equivalent polynomial with relative approximation degree one. 


\subsection{Relative approximation degree in the algebraic type setting} 
Since $(K(X)|K,w)$ is immediate, there exists, by \cite[Theorem 1]{Kaplansky1942}, a pseudo-convergent sequence $\mathcal{C}:= \{z_\nu\}_{\nu<\l} \subset K$ without a limit in $K$ and with $X$ as a limit. Thus the elements \(z_\nu\) provide increasingly accurate approximations to \(X\). For a polynomial \(f\in K[X]\), the relative approximation degree measures the rate at which \(f(z_\nu)\) approaches \(f(X)\). 

\begin{Definition}\label{Defn rel appr deg}
	Let $f\in K[X]$. Assume that there exists an ordinal $\nu_0<\l$, a positive integer $h$ and $\b\in vK$ such that
	\[ w(f(X)-f(z_\nu)) = \b + hw(X-z_\nu)  \]
	for all $\nu_0 <\nu < \l$. Then $h$ is called the \textbf{relative approximation degree of $f$ in $X$} and is denoted by $h_K(X:f)$. Moreover, $\b$ is called the \textbf{relative approximation constant of $f$ in $X$} and denoted by $\b_K(X:f)$.
\end{Definition} 
The existence of these invariants depends on the behaviour of \(\C\). It follows from \cite[Lemmas 1 and 5]{Kaplansky1942} that for any polynomial $f\in K[X]$, the sequence $\{vf(z_\nu)\}_{\nu<\l}$ is either ultimately stable, or ultimately strictly increasing. If the former holds for every polynomial $f$ over $K$, we say that $\mathcal{C}$ is of \textbf{transcendental type}. Otherwise, we say that $\mathcal{C}$ is of \textbf{algebraic type}. If $\mathcal{C}$ is of algebraic type, then a monic polynomial $Q(X)$ of least degree whose value is not ultimately fixed by $\mathcal{C}$ is said to be an \textbf{associated minimal polynomial to $\mathcal{C}$}. Observe that $Q$ is necessarily irreducible over $K$. We set 
\[ \deg\mathcal{C}:= \begin{cases}
	\deg Q &\quad\text{if $\mathcal{C}$ is of algebraic type},\\
	\infty &\quad\text{if $\mathcal{C}$ is of transcendental type}.
\end{cases}  \]  
The following results are collected from \cite{KuhlmannVlahu2014}:
\begin{enumerate}[label=(\roman*)]
	\item The invariants $h_K(X:f)$ and $\b_K(X:f)$ exist whenever $\deg f \leq \deg \mathcal{C}$ \cite[Lemma 5.2]{KuhlmannVlahu2014}, and also whenever \( \{vf(z_\nu)\} \) is ultimately strictly increasing \cite[Lemma 5.4]{KuhlmannVlahu2014}. Consequently, they exist for every polynomial over $K$ if $\mathcal{C}$ is of transcendental type.
	\item The remaining case, that is when $\mathcal{C}$ is of algebraic type and \( \{vf(z_\nu)\} \) is ultimately stable, remained open.
\end{enumerate}
The difficulty in the algebraic type case is the following. Let 
\[  f(X) = \sum_{i=0}^{n} f_i(z_\nu) (X-z_\nu)^i  \]
be the Taylor expansion of \(f(X)\) at \(z_\nu\), where \(f_i(X)\) denotes the \(i\)-th Hasse-Schmidt derivative. In the transcendental type case, the values of the coefficients \(vf_i(z_\nu)\) are ultimately stable along \(\C\), and hence the dominant Taylor monomial can be identified directly. This need not hold when \(\C\) is of algebraic type.

\pars The first objective of this paper is therefore to complete the theory of relative approximation degrees and constants by establishing the existence of \(h_K(X:f)\) and \(\b_K(X:f)\) in the setup of (ii). We do this over \textit{henselian} fields (see Section \ref{Section Z(f) has no limits}). Considering the \(Q\)-expansion of \(f\): 
\[  f = \sum_{i=0}^{r} f_iQ^i, \]
we overcome the earlier difficulty by a careful analysis of the \(Q\)-free term \(f_0\) and the remaining \(Q\)-divisible term along the pseudo-convergent sequence \(\C\).

\pars In the situation when \( \{ vf(z_\nu) \} \) is ultimately strictly increasing, we observe that \( h_K(X:f) \) coincides with the number of roots of \(f\), counted with multiplicity, which are also limits of \(\C\) (Theorem \ref{Thm h(f) = j_w*(f) when f has limits}). This motivates us to study the relative approximation degree in terms of the \textit{$j$-invariants} (see \cite{Dutta2024Invariant}) of certain auxiliary valuations arising naturally from the pseudo-convergent sequence \(\C\). A crucial property of the \(j\)-invariant is that it encodes information about the dominant term in the corresponding Taylor expansion (Corollary \ref{Coro j(f) when deg(Q)=1}). This allows us to reinterpret the relative approximation degree purely in terms of Taylor expansions of \(f\) along \(\C\) (Theorem \ref{Thm central j-invariant characterization}):

\begin{Theorem}
	Assume that \((K,v)\) is a henselian valued field and take \(f\in K[X] \setminus K\). Then there exists $\nu_0<\l$ such that for all \(\nu>\nu_0\), considering the Taylor expansion
	\[ f(X) - f(z_\nu) = \sum_{i=1}^{n} c_i (X-z_\nu)^i,  \]
	we have
	\begin{equation*}
		h_K(X:f) = j_{w_\nu}(f(X) - f(z_\nu)) = \max\left\{ i \ \middle|\ w_\nu (f(X) - f(z_\nu)) = w \left( c_i (X-z_\nu)^i \right) \right\}, 
	\end{equation*}
	and
	\begin{equation*}
		\b_K(X:f) = vc_h \text{ where } h=h_K(X:f).
	\end{equation*}
\end{Theorem}
Here \(w_\nu\) denotes the monomial valuation determined by \(z_\nu\). In particular, this theorem shows that relative approximation degrees and constants exist for
every polynomial in the algebraic type setting over a henselian base field. Combining this description with the degree theorem of Kuhlmann-Vlahu gives the following extension of their result:

\begin{Theorem}\label{Thm K(X)^h:K(f)^h leq h}
	Assume that $(K,v)$ is a henselian valued field and take $f\in K[X]\setminus K$. Then
	\[ [K(X)^h : K(f(X))^h] \leq h_K(X:f).  \]
\end{Theorem} 


\subsection{Artin-Schreier reduction over perfect fields} In Kuhlmann's framework, the corresponding pseudo-convergent sequence is necessarily of transcendental type \cite[Lemma 4.6]{Kuhlmann2019EliminationII}. Given an arbitrary polynomial \( f \), this allows him to construct a suitable \textit{normal form} for \( f \) \cite[Lemma 4.2]{Kuhlmann2019EliminationII}, from which he obtains an Artin-Schreier equivalent polynomial whose relative approximation degree is one. However, this method does not extend to the case where $\mathcal{C}$ is of algebraic type, since the construction of the normal forms relies fundamentally on the pseudo-convergent sequence to be of transcendental type. 

\pars In the general scenario, a straightforward operation yields an Artin-Schreier equivalent polynomial \(F\in K[X]\) satisfying 
\[  h_K(X:F) \in \{1,p\}  \]
(see Proposition \ref{Prop h(X:F) leq p}). Moreover, 
\[ h_K(X:F) = 1 \text{ whenever } w\left( f(X) - f(z_\nu) \right)  > 0  \text{ eventually}.  \]
However, the complementary case proves to be the essential difficulty in obtaining henselian rationality over perfect fields. In Section \ref{Sect example}, we construct an example illustrating the following:
\begin{Theorem}
	There exists an explicit immediate extension of function fields \((K(X)|K,w)\) of transcendence degree one such that \((K,v)\) is a henselian perfect valued field of rank one, and a polynomial \(f(X)\in K[X]\) such that 
	\[ h_K(X:f) = p \le h_K(X:F) \text{ for all } F\equiv f \bmod \P(K[X]).  \]
	Furthermore, after fixing an extension of \(w\) to \(\overline{K(X)}\), there exists a polynomial \(G\in K[X]\) such that 
	\[ G\equiv f \bmod \P(K(X)^h) \text{ and } h_K(X:G)=1.  \]
	Consequently, if \(\th^p-\th = f\), then the extension
	\[ \left( K(X,\th)|K, w \right) \text{ is henselian rational}.   \]
\end{Theorem}
The two conclusions of the theorem highlight an important distinction: the obstruction concerns representatives of the Artin-Schreier class of \(f\) modulo \(\P(K[X])\); it does not manifest after passage to the larger equivalence class modulo \(\P(K(X)^h)\). 

\pars Our construction of the polynomial \(G\) suggests a possible route bypassing the obstruction. Namely, one may first seek a polynomial \(g\) such that \(w(g-f)>0\), so that \(g\equiv f \bmod \P(K(X)^h)\), and then attempt to find a representative of the class of \(g\) modulo \(\P(K[X])\) having relative approximation degree one. In our example, the required perturbation is furnished by a suitable scalar multiple of the associated minimal polynomial \(Q\). However, we do not presently know how to produce such a perturbation in general. We therefore pursue a complementary approach in Section \ref{Section henselian rationality over IC}: enlarging the constant field inside the henselization until the relevant pseudo-convergent sequence becomes of transcendental type. 


\subsection{Henselian rationality over the relative algebraic closure} In the rank-one perfect field setting, henselian rationality can be pulled down through \textit{finite} extensions (Lemma \ref{Lemma henselian ratioanlity pull down finite extns}). This observation suggests replacing \(K\) temporarily by a suitable algebraic extension contained in \(F^h\), proving henselian rationality over the enlarged constant field, and then investigating whether the resulting extension of \(K\) is finite.

\pars The natural fields arising in this context are the \textit{implicit constant fields} introduced by Kuhlmann \cite{Kuhlmann2004BadPlaces}. For a separating transcendental element \(X\in F\setminus K\), set
\[ I:=IC(K(X)|K,w), \]
the relative algebraic closure of \(K\) in \(K(X)^h\). We also set
\[ L:=IC(F|K,w), \]
the relative algebraic closure of \(K\) in \(F^h\). Unlike \(I\), the field \(L\) depends only on the extension \((F|K,w)\) and not on the choice of \(X\).

\pars Our principal observation is that whenever \(K\) equals its absolute ramification field, the pseudo-convergent sequence corresponding to the extension \((I(X)|I,w)\) is of transcendental type (Proposition \ref{Prop pcs tr type over IC}). Thus passage from \(K\) to \(I\) removes the algebraic type obstruction arising in the earlier framework. This restores precisely the hypothesis required in Kuhlmann's normal-form argument. As a consequence, we obtain in Proposition \ref{Prop henselian rationality over IC} that
\[ (F^h|I,w) \text{ is henselian rational}.  \]
We further observe that \(L|I\) is a finite extension. Thus the finiteness of \(I\) over \(K\) is a property of the extension \((F|K,w)\) itself, and is independent of the choice of the separating transcendental element \(X\). Combining these observations, we arrive at the following central result: 

 \begin{Theorem}\label{Thm henselian rationality over IC}
 	Assume that \((K,v)\) is a perfect valued field of rank one. Furthermore, assume that \(K\) equals its absolute ramification field. Set \(L\) to be the relative algebraic closure of \(K\) in \(F^h\). Then,
 	\[  (F^h|L,w) \text{ is henselian rational}.  \]
 	Moreover, 
 	\[  [L:K] < \infty \Longrightarrow (F|K,w) \text{ is henselian rational}. \]
 \end{Theorem}
Thus after passage to the absolute ramification field, the problem reduces to understanding whether henselian rationality can be pulled down through the algebraic extension \(L|K\). A successful resolution would yield that henselian rationality holds over rank-one perfect valued fields coinciding with their absolute ramification field. This would be an analogue of Step \ref{Step B} in Kuhlmann's program. Observe that this is a strict generalization, since there are non-algebraically closed perfect valued fields of rank one equalling their absolute ramification fields (Example \ref{Eg K not alg closed}). 

\pars  As observed in Theorem \ref{Thm henselian rationality over IC}, the pull-down principle survives when \(L|K\) is a finite extension. In Proposition \ref{Prop Type II}, we verify this finiteness criterion for the class of Type II extensions introduced in \cite{Dutta2023MathNach}. The extension \((K(X)|K,w)\) is said to be of Type II if \(X\) is a limit of a Cauchy sequence in the algebraic closure \(\overline{K}\). This yields a general henselian rationality theorem for immediate function fields admitting a separating transcendental element of Type II.

\pars We conclude with a brief outline of the paper. Section \ref{Section Prelims} develops the necessary background on pseudo-convergent sequences, minimal pairs of definition and the \(j\)-invariant. Section \ref{Section Z(f) has limits} establishes the characterization of \(h_K(X:f)\) and \(\b_K(X:f)\) in the algebraic type setting when the value of \(f\) is ultimately strictly increasing along \(\C\), while the remaining case is considered in Section \ref{Section Z(f) has no limits}. In Section \ref{Section K(X)^h:K(f)^h leq h}, we derive the henselian degree bound. Section \ref{Section A-S reduction} studies the effect of Artin-Schreier equivalence over perfect fields. Section \ref{Sect example} constructs the aforementioned example showing that the bound \(p\) is sharp, and then proves that the associated Artin-Schreier function field is nevertheless henselian rational. Finally, Section \ref{Section henselian rationality over IC} develops the implicit constant field approach to henselian rationality and proves the finite descent and Type II results described above. Section \ref{Section open problems} isolates the remaining Type I pull-down problem.

\section{Preliminaries} \label{Section Prelims}
We fix some notations that will be used throughout Sections \ref{Section Prelims}--\ref{Section A-S reduction}. Let \( (K(X)|K,w) \) be an immediate extension of rational function fields. Fix an extension of \(w\) to \(\overline{K}(X)\), denoted again by \(w\). Its restriction to \(\overline{K}\) will still be denoted by \(v\). Given any polynomial \( f\in K[X] \), denote by \( \Z(f) \) the \textit{multiset} of all roots of \(f\). For any multiset \(S\), we denote its cardinality by \(\# S\).

\subsection{Pseudo-convergent sequences}\label{subsection pcs} A sequence \( \{z_\nu\}_{\nu<\l} \) in \( (K,v) \), where \( \l \) is a limit ordinal, is said to be a \textbf{pseudo-convergent sequence} if 
\[  v\left( z_{\nu} - z_\mu  \right) < v\left( z_\mu - z_\rho \right)  \]
for all \( \nu<\mu<\rho<\l \). It follows from the triangle inequality that
\[  v\left( z_{\nu} - z_\mu  \right) = v\left( z_\nu - z_{\nu + 1}  \right)  \]
for all \( \nu<\mu<\l \). 

\pars Since $(K(X)|K,w)$ is immediate, there exists, by \cite[Theorem 1]{Kaplansky1942}, a pseudo-convergent sequence $\mathcal{C}:= \{z_\nu\}_{\nu<\l} \subset K$ without a limit in $K$ and with $X$ as a limit. Set 
\[ \g_\nu:= v(z_{\nu} - z_{\nu +1})  \]
for all $\nu<\l$. An element $Y\in\overline{K(X)}$ is said to be a \textbf{limit} of $\mathcal{C}$ if and only if 
\[ w(Y-z_\nu) = \g_\nu \text{ for all } \nu<\l.  \]
From the triangle inequality, it then follows that 
\[ Y\in\Lim\mathcal{C} \Longleftrightarrow w(Y-X) \geq \g_\nu \text{ for all } \nu<\l. \]
We denote the set of all limits of $\mathcal{C}$ in $\overline{K(X)}$ by $\Lim\mathcal{C}$. The following observation is immediate from the monotonicity of $\{\g_\nu\}_{\nu<\l}$:

\begin{Lemma}\label{Lemma Y not limit eqv condns}
	Take any $Y\in\overline{K(X)}$. Then the following are equivalent:
	\begin{enumerate}[label=(\roman*)]
		\item $Y\notin\Lim\mathcal{C}$,
		\item the sequence $\{w(Y-z_\nu)\}_{\nu<\l}$ is ultimately constant,
		\item there exists an ordinal $\nu_0<\l$ such that $w(Y-z_\nu) < \g_\nu$ for all $\nu>\nu_0$.
	\end{enumerate}
\end{Lemma}

As a corollary we obtain the following result which was proved in \cite[Lemma 9.1]{Dutta2021}:

\begin{Lemma}\label{Lemma vf(z) increasing iff Z(f) has limits}
	Take any $f\in K[X]$. Then $\{vf(z_\nu)\}_{\nu<\l}$ is ultimately strictly increasing if and only if $\Z(f)\sect\Lim\mathcal{C}\neq\emptyset$. 
\end{Lemma}


\subsection{Minimal pairs and the \(j\)-invariant} We take a brief detour at this juncture to give a quick introduction to the theory of minimal pairs and the \(j\)-invariant. Any extension \(w\) of \(v\) to \(K(X)\) satisfies the \textbf{Abhyankar inequality}:
\[ \dim_\QQ \left( wK(X)/vK \right) + \trdeg [K(X)w:Kv] \leq 1.  \]
This is a consequence of Theorem 1 of \cite[\S 10.3, Chapter VI]{Bourbaki1989}. If \(w\) admits equality in the above inequality, then it is determined by a pair \( (a,\g) \) in the following sense: take an extension of \( w \) to \( \overline{K}(X) \), which we again denote by \(w\). Then there exists \(a\in\overline{K}\) and \(\g\in w\overline{K}(X)\) such that for any polynomial \(f(X)\in\overline{K}[X]\), we have 
\[ f(X) = \sum_{i=0}^{n} c_i (X-a)^i \Longrightarrow wf = \min\{ vc_i + i\g \}.  \]
In this case we write \(w = v_{a,\g}\) and say that \( (a,\g) \) is a \textbf{pair of definition for \( (K(X)|K,w) \)}. The valuation \(w\) may admit several pairs of definition. Indeed, we have the following from \cite[Proposition 3]{AlexandruZaharescu1988}:
\begin{Proposition}\label{Prop pair of defn}
	Take \( a_1,a_2\in\overline{K} \) and \( \g_1, \g_2 \) in some ordered abelian group containing \(v\overline{K} \). Then,
	\[ v_{a_1,\g_1} = v_{a_2,\g_2} \Longleftrightarrow \g_1 = \g_2 \text{ and } v(a_1 - a_2)\geq\g_1.  \]
\end{Proposition}
In light of the above proposition, we define a pair of definition \( (a,\g) \) to be a \textbf{minimal pair of definition for \( (K(X)|K,w) \)} if \(a\) has minimal degree over \(K\) among all pairs of definition, i.e., 
\[ v(a-b)\geq\g \Longrightarrow [K(a):K] \leq [K(b):K].  \]

\pars We now fix a \textit{minimal} pair of definition \( (a,\g) \) for \( (K(X)|K,w) \). Let \(Q\) be the minimal polynomial of \(a\) over \(K\) and take \( f\in K[X] \). Then we have a unique expansion, called the \(Q\)-expansion
\begin{equation}\label{eqn Q-expansion}
	f = \sum_{i=0}^{r} f_i Q^i,
\end{equation}
where \( f_i \in K[X]\) with \( \deg f_i < \deg Q \). The following theorem appears in \cite[Theorem 2.1]{AlexandruPopescuZaharescu1988}:

\begin{Theorem}\label{Thm w = v_Q}
	\( wf = \min\{ wf_i + iwQ \} = \min\{ vf_i(a) + iwQ \}.  \)
\end{Theorem}

For any polynomial \( f\in K[X] \), we define
\[ j_w(f) := \#\{ z\in \Z(f) \mid v(a-z)\geq\g \}.  \]
This operator \( j_w(-) \) is said to be the \textbf{\(j\)-invariant corresponding to \(w\)}. The invariance of \(j_w(-)\) was established in \cite[Theorem 3.1]{Dutta2024Invariant}. In particular, it is independent of both the chosen extension of \(w\) to \(\overline{K}(X)\) and the chosen minimal pair of definition. 

\pars The \(j\)-invariant encodes valuable ramification theoretic information, as well as providing important descriptions of the valuation \(w\). In this paper, we are primarily interested in the latter properties. The following observation is a consequence of \cite[Propositions 3.4 and 2.14]{Dutta2024Invariant} and will play a key role in the sequel: 

\begin{Theorem}\label{Thm j(f)/j(Q)}
	Take a polynomial \( f \in K[X] \). Then,
	\[ j_w(Q) \text{ divides } j_w(f).  \]
	More precisely, considering the expansion (\ref{eqn Q-expansion}), we have
	\[  \dfrac{j_w(f)}{j_w(Q)} = \max \left\{ i \mid wf = wf_i + iwQ \right\}.  \]
\end{Theorem} 

As a consequence, in the case when \(a\in K\), we obtain the following:

\begin{Corollary}\label{Coro j(f) when deg(Q)=1}
	Assume that \(a\in K\). Take \( f\in K[X] \) and consider the expansion 
	\[ f = \sum_{i=0}^{n} c_i (X-a)^i.  \]
	Then, 
	\[ wf = \min\{ vc_i + i\g \}.  \]
	Moreover, 
	\[ j_w(f) = \max\left\{ i \mid wf = vc_i + i\g \right\}. \]
\end{Corollary}

This next property plays a significant role in building a connection between the \(j\)-invariant and the relative approximation degree. 

\begin{Proposition}\cite[Corollary 3.6]{Dutta2024Invariant}\label{Prop in(f)=in(g) implies j(f)=j(g)}
	Assume that \( w(f-g) > wf = wg \). Then \( j_w(f) = j_w(g) \).
\end{Proposition}


\subsection{Monomial valuations} We now return to our original setup. For each $\nu<\l$, we define the monomial valuation
\[ w_{\nu}:= v_{z_\nu, \g_\nu}.  \]
Now $z_\nu$ being in $K$ implies that $(z_\nu,\g_\nu)$ is a minimal pair of definition for $(K(X)|K,w_\nu)$. Moreover, we have $v(z_\mu-z_\nu) = \g_\nu$ for all $\mu>\nu$. It follows that
\begin{equation}\label{eqn mpd for w_nu}
	(z_\mu,\g_\nu) \text{ is a minimal pair of definition for } (K(X)|K,w_\nu) \text{ for all } \nu\leq\mu<\l.
\end{equation}
If $\mathcal{C}$ is of algebraic type with an associated minimal polynomial $Q(X)$, set
\[ \g:= \sup\{\g_\nu\}_{\nu<\l},  \]
where the supremum $\g$ is taken in the Dedekind-MacNeille completion of $vK$. We refer the reader to \cite[Section 5]{Schroder2003OrderedSets} for an introduction to the theory of Dedekind-MacNeille completions. Since $X\in\Lim\mathcal{C}$ and $\mathcal{C}$ is of algebraic type, $\mathcal{C}$ is not a Cauchy sequence, and hence $\g\neq\infty$. We take $a\in \Z(Q)\sect\Lim\mathcal{C}$ and define 
\[ w^*:= v_{a,\g}.  \]
That such an $a$ exists is guaranteed by Lemma \ref{Lemma vf(z) increasing iff Z(f) has limits}. Since $a\in\Lim\C$ and $\g>\g_\nu$ for all $\nu<\l$, we have that for any $b\in\overline{K}$,
\begin{equation}\label{eqn limit iff > g_nu}
	v(a-b)\geq\g\Longleftrightarrow b\in\Lim\mathcal{C}.
\end{equation}
In other words, 
\begin{equation}\label{eqn (b,g) pair of defn iff b limit}
	(b,\g) \text{ is a pair of definition for } (K(X)|K,w^*) \text{ if and only if } b\in\Lim\C.
\end{equation}
From the minimality of $Q$, we then conclude that
\begin{equation}\label{eqn mpd for w*}
	(a,\g) \text{ is a minimal pair of definition for } (K(X)|K,w^*).
\end{equation}

By Proposition \ref{Prop pair of defn}, \( (a,\g_\nu) \) is also a pair of definition for \( w_\nu \) for each \( \nu<\l \). Thus \( w^* \) can be regarded as the limit of the valuations \(\{ w_\nu \}_{\nu<\l}\).

\begin{Lemma}\label{Lemma w_nu f = vf(z_nu)}
	Take \( f \in K[X] \). Then there exists \( \nu_0 < \l \) such that \( w_\nu f = v f(z_\nu) \) for all \( \nu > \nu_0 \).
\end{Lemma}

\begin{proof}
	Writing \( f = \prod_{i=1}^{n} (X-z_i) \), we choose \( \nu_0 < \l \) such that 
	\[ \g_{\nu_0} > \max \left\{ w(X-z_i) \mid z_i \notin \Lim \C  \right\},  \]
	where the maximum over the empty set is understood to be \( -\infty \). Such a choice is possible by Lemma \ref{Lemma Y not limit eqv condns}. Fix some \( \nu > \nu_0 \). Observe that 
	\[ w_\nu (X-z_i) = \g_\nu = v(z_\nu - z_i) \text{ whenever } z_i \in \Lim \C.  \]
	Otherwise, \( \g_\nu > w(X-z_i) \) and hence 
	\[ w_\nu (X-z_i) = v (z_\nu - z_i).   \]
	The lemma now follows.
\end{proof}


\section{The case $\Z(f)\sect\Lim\mathcal{C}\neq\emptyset$}\label{Section Z(f) has limits}
Throughout Sections \ref{Section Z(f) has limits}--\ref{Section A-S reduction}, we assume that $\C$ is of algebraic type. We fix an associated minimal polynomial $Q$ and set
\[   d:= \deg Q = \deg\C.  \]

\subsection{Relative approximation degree and the \(j\)-invariant} We first relate the relative approximation degree to the $j$-invariant corresponding to $w^*$. 

\begin{Theorem}\label{Thm h(f) = j_w*(f) when f has limits}
	Let $f \in K[X]$ such that $\Z(f)\sect\Lim\C\neq\emptyset$. Then
	\begin{equation}\label{eqn h(f) = j_w*(f)}
		h_K(X:f) = \# \Z(f)\sect\Lim\C = j_{w^*}(f).
	\end{equation}
Moreover, 
\begin{equation}\label{eqn h(Q)| h(f)}
	h_K(X:Q) \text{ divides } h_K(X:f).
\end{equation}
\end{Theorem}

\begin{proof}
	Write 
	\[ f(X) = c(X-a_1)\dotsc (X-a_n),   \]
	where $c\in K$ and assume that 
	\[ \Z(f)\sect\Lim\C = \{a_1, \dotsc , a_j\}. \]
	Since $\Z(f)\sect\Lim\C\neq\emptyset$, employing \cite[Corollary 7.1]{KuhlmannVlahu2014}, we can choose an ordinal \( \nu_0 \) such that
	\[ vf(z_\nu) = w(f(X)-f(z_\nu)) = \b_K(X:f) + h_K(X:f)\g_\nu \text{ for all } \nu>\nu_0.  \]
	By Lemma \ref{Lemma Y not limit eqv condns}, increasing $\nu_0$ if necessary we can further assume that the sequence 
	\[  \{v(z_\nu-a_i)\}  \]
	is constant for all $i>j$ and $\nu>\nu_0$. Denote this constant value by $\b_i$. It follows that
	\[ vf(z_\nu) = vc + j\g_\nu + \sum_{i>j} \b_i \text{ for all } \nu>\nu_0.   \] 
	Comparing the above two expressions yields
	\begin{equation}\label{eqn h_K(X:f)-j}
		\left(h_K(X:f)-j\right)\g_\nu = vc+\sum_{i>j}\b_i - \b_K(X:f).
	\end{equation}
If $h_K(X:f)\neq j$, then the left hand side of (\ref{eqn h_K(X:f)-j}) is monotonic in $\nu$, whereas the right hand side is constant. This contradiction implies that $h_K(X:f)=j$. Hence
\[  h_K(X:f) = j = \#\Z(f)\sect\Lim\C, \]
which proves the first equality. Moreover, it follows from (\ref{eqn limit iff > g_nu}) that
\[ v(a-a_i)\geq\g \text{ if and only if } a_i\in\Lim\C.  \]
This is equivalent to $a_i\in\{a_1, \dotsc , a_j\}$, completing the proof of the second equality. In particular, we obtain that $h_K(X:Q) = j_{w^*}(Q)$. The final assertion is now a direct consequence of (\ref{eqn mpd for w*}) and Theorem \ref{Thm j(f)/j(Q)}. 
\end{proof}

Since our choice of the associated minimal polynomial $Q$ was arbitrary, we obtain the following as an immediate consequence of (\ref{eqn h(Q)| h(f)}): 

\begin{Corollary}
	Let $Q^\prime$ be another associated minimal polynomial to $\C$. Then
	\[  h_K(X:Q) = h_K(X:Q^\prime).  \]
\end{Corollary}

Analogous to Theorem \ref{Thm h(f) = j_w*(f) when f has limits}, we can also relate the relative approximation degree to the $j$-invariant associated to $w_\nu$ for all large enough $\nu$. 

\begin{Proposition}\label{Prop h(f) = j_w_nu(f) when f has limits}
	Let $f \in K[X]$ such that $\Z(f)\sect\Lim\C\neq\emptyset$. Then there exists $\nu_0<\l$ such that
	\begin{equation}\label{eqn h(f) = j w_nu (f)}
		h_K(X:f) = j_{w_\nu}(f) 
	\end{equation}
	for all \( \nu>\nu_0 \). Moreover, if
	\[ f(X) - f(z_\nu) = \sum_{i=1}^{n} c_i (X-z_\nu)^i  \]
	for some \( \nu>\nu_0 \), then
	\begin{equation}\label{eqn beta(f) = vc_h}
		\b_K(X:f) = vc_h \text{ where } h=h_K(X:f).
	\end{equation}
\end{Proposition}

\begin{proof}
	Write $f$ as in the proof of Theorem \ref{Thm h(f) = j_w*(f) when f has limits}. By Lemma \ref{Lemma Y not limit eqv condns}, we can choose $\nu_0<\l$ such that for all $\nu>\nu_0$ and $i>j$, we have 
	\[ v(z_\nu-a_i) < \g_\nu.  \]
	Since $\{a_1, \dotsc , a_j\} \subseteq\Lim\C$, we also have $v(z_\nu-a_i) = \g_\nu$ for all $i\leq j$. Hence, by (\ref{eqn h(f) = j_w*(f)}), 
	\[ j_{w_\nu}(f) = j = j_{w^*}(f) = h_K(X:f).  \]
	Now consider the expansion 
	\[ f(X) - f(z_\nu) = \sum_{i=1}^{n} c_i (X-z_\nu)^i  \]
	and rename $h= j = h_K(X:f)$. Applying Corollary \ref{Coro j(f) when deg(Q)=1} together with (\ref{eqn mpd for w_nu}) and (\ref{eqn h(f) = j w_nu (f)}), we obtain that
	\[ w_\nu f = vc_h + h\g_\nu.  \]
	Observe that our choice of \(\nu_0\) ensures $w_\nu f = vf(z_\nu)$ for all $\nu>\nu_0$. Thus the above equation can be rewritten as
	\[ vf(z_\nu) = vc_h + h\g_\nu.  \]
	On the other hand, \cite[Corollary 7.1]{KuhlmannVlahu2014} yields that 
	\[ vf(z_\nu) = \b_K(X:f) + h\g_\nu.  \]
	Comparing the two expressions, we arrive at (\ref{eqn beta(f) = vc_h}). 
\end{proof}

It is implicit in the proof of the above theorem that
	\[  w_\nu \left( f(X)-f(z_\nu) \right) = vc_h+h\g_\nu \leq vc_i + i\g_\nu \text{ for all } i\neq h.  \]
	Since \( j_{w_\nu}(f) = h \), Corollary \ref{Coro j(f) when deg(Q)=1} yields that the above inequality is strict for \( i>h \), and hence
	\[  j_{w_\nu} \left( f(X)-f(z_\nu) \right) = h \]
	for all \( \nu>\nu_0 \). We have thus arrived at the following auxiliary result, which is a slight weakening of Proposition \ref{Prop h(f) = j_w_nu(f) when f has limits}, tailored to our later applications: 
	
\begin{Theorem}\label{Thm j-invariant characterization}
	Let $f\in K[X]$ such that $\Z(f)\sect\Lim\C\neq\emptyset$. Then there exists $\nu_0<\l$ such that
	\begin{equation}\label{eqn h(f) = j w_nu (f - f(z))}
		h_K(X:f) = j_{w_\nu}(f(X) - f(z_\nu)) 
	\end{equation}
	for all \( \nu>\nu_0 \). Moreover, if
	\[ f(X) - f(z_\nu) = \sum_{i=1}^{n} c_i (X-z_\nu)^i  \]
	for some \( \nu>\nu_0 \), then
	\begin{equation*}
		\b_K(X:f) = vc_h \text{ where } h=h_K(X:f).
	\end{equation*}
\end{Theorem}	

The following corollary is immediate in light of Theorem \ref{Thm h(f) = j_w*(f) when f has limits}:

\begin{Corollary}
	Let \( f, g \in K[X] \) with \( \Z(f)\sect\Lim\C\neq\emptyset \) and \( \Z(g)\sect\Lim\C\neq\emptyset \). Then,
	\[  h_K(X:fg) = h_K(X:f) + h_K(X:g).  \]
\end{Corollary}


\subsection{The henselian core of \(Q\)}

\begin{Definition}
	Let $\Z(Q) = \{ a_1, \dotsc , a_d \}$ with $a=a_1$. Moreover, assume that 
	\[ \Z(Q)\sect\Lim\C = \{a_1, \dotsc , a_j\}. \]
	We define the \textbf{henselian core} of $Q(X)$ as
	\[ Q^h(X):= \prod_{i=1}^{j} (X-a_i).  \] 
\end{Definition}

In view of (\ref{eqn h(f) = j_w*(f)}), we observe that
\begin{equation}\label{eqn deg Q^h = h(Q)}
	\deg Q^h = h_K(X:Q).
\end{equation}
The following result provides justification for the nomenclature. 

\begin{Proposition}\label{Prop Q^h irr over K^h}
   The polynomial $Q^h$ lies in $K^h[X]$ and is irreducible over $K^h$.
\end{Proposition}

\begin{proof}
	By (\ref{eqn mpd for w*}), the pair $(a,\g)$ is a minimal pair of definition for $(K(X)|K,w^*)$. Hence, \cite[Corollary 2.8]{DuttaGhosh2025} implies that
	\begin{equation}\label{eqn mpd for w^* over K^h}
		(a,\g) \text{ is a minimal pair of definition for } (K^h(X)|K^h,w^*).
	\end{equation}
	Let $Q^\prime$ denote the minimal polynomial of $a$ over $K^h$. It then follows from \cite[Proposition 2.9]{DuttaGhosh2025} that
	\begin{equation}\label{eqn j leq deg F}
		\deg Q^h = j = j_{w^*}(Q) = j_{w^*}(Q^\prime)\leq \deg Q^\prime.
	\end{equation}
Now take any $b\in\Z(Q^\prime)$. Then $b = \s a$ for some \( \s \) in the decomposition group $G^d(\overline{K}|K,v)$. Since $\s$ is valuation preserving, we have
\[  v(b-z_\nu) = v\s(a-z_\nu) = v(a-z_\nu) = \g_\nu \text{ for all } \nu<\l.  \]
Hence $b$ is also a limit of $\C$. As a consequence, 
\begin{equation}\label{eqn Z(F) subset Lim C}
	\Z(Q^\prime)\subseteq \Z(Q)\sect\Lim\C = \Z(Q^h).
\end{equation}
From (\ref{eqn j leq deg F}) and (\ref{eqn Z(F) subset Lim C}), we conclude that $Q^\prime=Q^h$, completing the proof. 
\end{proof}

The following consequence generalizes \cite[Corollary 9.2]{KuhlmannVlahu2014}. 

\begin{Corollary}\label{Coro h(Q) = deg(Q) iff Q irr}
	$h_K(X:Q) = \deg Q$ if and only if $Q$ is irreducible over $K^h$.
\end{Corollary}

\begin{proof}
	Since $h_K(X:Q) = \deg Q^h$ by (\ref{eqn h(f) = j_w*(f)}), we have $h_K(X:Q) = \deg Q$ if and only if $Q=Q^h$. The assertion therefore follows immediately from Proposition \ref{Prop Q^h irr over K^h}. 
\end{proof}

\begin{Remark}
	Corollary \ref{Coro h(Q) = deg(Q) iff Q irr} implies that $h_K(X:Q) = \deg Q$ whenever $K$ is dense in its henselization $K^h$. Indeed, \cite[Corollary 3.2]{Dutta2023MathNach} yields $Q$ is irreducible over the completion $\widehat{K}$. In particular, this holds whenever $K$ is henselian or of rank one.  
\end{Remark}

At the other end of the spectrum lies the case when $h_K(X:Q) = 1$, characterized by the following proposition:

\begin{Proposition}\label{Prop h(Q)=1}
	The following are equivalent:
	\begin{enumerate}[label = (\roman*)]
		\item \(h_K(X:Q) = 1\),
		\item \(\Z(Q)\sect\Lim\C\) contains an element of \(K^h\), 
		\item \(\Lim\C\) contains an element of \(K^h\).
	\end{enumerate}
\end{Proposition}

\begin{proof}
	Since \(h_K(X:Q) = \deg Q^h\) by (\ref{eqn deg Q^h = h(Q)}), Proposition \ref{Prop Q^h irr over K^h} immediately yields \((i) \Longrightarrow (ii)\). 
	
	The implication \((ii)\Longrightarrow(iii)\) is tautological. 
	
	We now assume that \(\Lim\C\) contains an element of \(K^h\), say \(b\). Then \((b,\g)\) is a minimal pair of definition for \((K^h(X)|K^h,w^*)\). On the other hand, \( (a,\g) \) is also a minimal pair of definition for \( (K^h(X)|K^h,w^*) \) by (\ref{eqn mpd for w^* over K^h}). It follows that \( a\in K^h \). By Proposition \ref{Prop Q^h irr over K^h}, the polynomial \(Q^h\) is the minimal polynomial of \(a\) over \(K^h\). Hence \( \deg Q^h = 1\). Since \(h_K(X:Q) = \deg Q^h\), we conclude that \( h_K(X:Q)=1 \).  
\end{proof}

The next result illustrates that the relative approximation degree is preserved under passage to the henselization. 

\begin{Corollary}
	Assume that \( h_K(X:Q)>1 \). Then \(Q^h\) is an associated minimal polynomial to \(\C\) viewed as a pseudo-convergent sequence in \(K^h\). Moreover,
	\[ h_K(X:Q) = \deg Q^h = h_{K^h}(X:Q^h).   \]
\end{Corollary}

\begin{proof}
	In view of Proposition \ref{Prop h(Q)=1}, the condition \( h_K(X:Q)>1 \) implies that \(\C\) has no limits in \(K^h\). Hence we can view \(\C\) as a pseudo-convergent sequence of algebraic type in \(K^h\). Let \(Q^\prime\in K^h[X]\) be an associated minimal polynomial to \(\C\) over \(K^h\), and choose \( b\in\Z(Q^\prime) \). Applying (\ref{eqn mpd for w*}) over \(K^h\), we obtain that \( (b,\g) \) is a minimal pair of definition for \( (K^h(X)|K^h,w^*) \). On the other hand, \( (a,\g) \) is also a minimal pair of definition for \( (K^h(X)|K^h,w^*) \) by (\ref{eqn mpd for w^* over K^h}). Thus
	\[ \deg Q^h = \deg Q^\prime. \]
	Since \( \{vQ^h(z_\nu)\}_{\nu<\l} \) is ultimately strictly increasing by Lemma \ref{Lemma vf(z) increasing iff Z(f) has limits}, we conclude that \(Q^h\) is an associated minimal polynomial to \(\C\). The equality \( \deg Q^h = h_{K^h}(X:Q^h) \) is now a consequence of Corollary \ref{Coro h(Q) = deg(Q) iff Q irr}. Together with (\ref{eqn deg Q^h = h(Q)}), this yields 
	\[ h_K(X:Q) = \deg Q^h = h_{K^h}(X:Q^h), \]
	thereby completing the proof.  
\end{proof}


\section{The case \( \Z(f)\sect\Lim\C = \emptyset \)}\label{Section Z(f) has no limits}
We now tackle the complementary case when no root of \(f\) is a limit of \(\C\). Since \( j_{w^*}(f) = 0 \) in this case, we instead focus on the values \( j_{w_\nu}(f) \). 

\subsection{The case \(\deg f \leq \deg\C\)} We begin with the case \( \deg f \leq \deg \C \). In this framework, \cite[Lemma 5.2]{KuhlmannVlahu2014} guarantees that \( \b_K(X:f) \) and \( h_K(X:f) \) are well-defined. The following proposition is an analogue of Theorem \ref{Thm j-invariant characterization}.

\begin{Proposition}\label{Prop h(f) = j_{w_nu} deg f < deg C}
	Assume that \( 0<\deg f \leq \deg\C \) and \( \Z(f)\sect\Lim\C=\emptyset \). Then the conclusions of Theorem \ref{Thm j-invariant characterization} hold true.  
\end{Proposition}

\begin{proof}
	We consider the Hasse-Schmidt derivatives \( f_i(X) \) of \(f\), satisfying the identity
	\begin{equation}\label{eqn Taylor expansion}
		f(X)-f(z) = \sum_{i=1}^{n} f_i(z)(X-z)^i  
	\end{equation}
	for all \( z\in K \). Since \(\deg f_i < \deg f \leq \deg \C\), the sequence of values \( \{vf_i(z_\nu)\}_{\nu<\l} \) is ultimately stable. Denote this eventual constant value by \( \b_i \). By \cite[Lemma 4]{Kaplansky1942}, there exist an ordinal \(\nu_0 \) and an integer \(1\leq h\leq n\) such that 
	\[ \b_h + h\g_\nu < \b_i + i\g_\nu   \]
	for all \(\nu>\nu_0\) and \( i\neq h \). Thus \( w\left( f(X)-f(z_\nu) \right) = \b_h + h\g_\nu \) for all \( \nu>\nu_0 \) by the triangle inequality. Hence
	\begin{equation}\label{eqn beta_K and h_K}
		\b_h = \b_K(X:f) \text{ and } h = h_K(X:f).
	\end{equation}
	Increasing \( \nu_0 \) if necessary, by Lemma \ref{Lemma w_nu f = vf(z_nu)}, we may additionally assume that
	\[ w_\nu f = vf(z_\nu) \text{ and } w_\nu f_i = vf_i(z_\nu) = \b_i \text{ for all } \nu>\nu_0.  \]
	Now consider the Taylor series expansion
	\[ f(X)-f(z_\nu) = \sum_{i=1}^{n} f_i(z_\nu)(X-z_\nu)^i.   \]
	 By (\ref{eqn beta_K and h_K}), we obtain that
	\[  \b_K(X:f) = \b_h = vf_h(z_\nu),  \]
	which proves the second assertion. Moreover,
	\[  w_\nu\left(f_h(z_\nu) (X-z_\nu)^h\right) = \b_h + h\g_\nu < \b_i + i\g_\nu = w_\nu\left(f_i(z_\nu) (X-z_\nu)^i\right)   \]
	for all \( i\neq h \). Hence,
	\[  w_\nu\left( f(X)-f(z_\nu) \right) = w_\nu\left(f_h(z_\nu) (X-z_\nu)^h\right) < w_\nu\left(f_i(z_\nu) (X-z_\nu)^i\right)   \]
	for all \( i\neq h \). Applying Corollary \ref{Coro j(f) when deg(Q)=1} together with (\ref{eqn mpd for w_nu}), we conclude that
	 \[ j_{w_\nu} \left(f(X) - f(z_\nu)\right) = h.   \]
	 Combining this with (\ref{eqn beta_K and h_K}) yields
	 \[ j_{w_\nu} \left(f(X) - f(z_\nu)\right) = h_K(X:f) \text{ for all } \nu>\nu_0.   \] 
\end{proof}

\begin{Remark}
	The proof of Proposition \ref{Prop h(f) = j_{w_nu} deg f < deg C} does not depend on the finiteness of \( \deg\C \). Hence the conclusions hold true also when \(\C\) is of transcendental type. 
\end{Remark}


\subsection{The case \( \deg f > \deg\C \)} Assume that \textit{\( (K,v) \) is henselian} and \( \deg f> \deg\C \). Consider (\ref{eqn Q-expansion}), the \( Q \)-expansion of \(f\):
\[ f = \sum_{i=0}^{r} f_iQ^i.  \]
We first show that \( h_K(X:f) \) and \( \b_K(X:f) \) are well-defined in this framework. To determine these quantities, we compare the contributions of the \textit{constant} part \(f_0\) and the higher \(Q\)-divisible part \(f-f_0\). Take \( \nu_0 \) satisfying the conditions of Theorem \ref{Thm j-invariant characterization} and Proposition \ref{Prop h(f) = j_{w_nu} deg f < deg C}. Set \[ f^\prime(X):= f(X) - f_0(X). \]
Since \(Q\) divides \(f^\prime\), the sequence \( \{vf^\prime(z_\nu)\} \) is ultimately strictly increasing. Then,
\[ w\left( f^\prime(X) - f^\prime(z_\nu)  \right) = \b^\prime + h^\prime\g_\nu \text{ for all } \nu> \nu_0,   \]
where \( \b^\prime = \b_K(X:f^\prime) \) and \( h^\prime = h_K(X:f^\prime) \). Moreover, (\ref{eqn h(Q)| h(f)}) and Corollary \ref{Coro h(Q) = deg(Q) iff Q irr} yield that
\[ d \text{ divides } h^\prime.  \]
On the other hand, the fact \(\deg f_0 < \deg\C\) implies that
\[ w\left( f_0(X) - f_0(z_\nu)  \right) = \b_0 + h_0\g_\nu \text{ for all } \nu> \nu_0,   \]
where \( \b_0 = \b_K(X:f_0) \) and \( h_0 = h_K(X:f_0) \). Observe that
\[  h_0 \leq \deg f_0 < d \leq h^\prime.  \]
As a consequence, we infer from the monotonicity of \( \{\g_\nu\} \) that the expressions \( \b_0 + h_0\g_\nu \) and \( \b^\prime + h^\prime\g_\nu \) can coincide for at most one index \(\nu\). Hence, after enlarging \( \nu_0 \) if necessary, we may further assume that either 
\begin{equation}\label{eqn condn Type A}
	\b_0 + h_0\g_\nu < \b^\prime + h^\prime\g_\nu \text{ for all } \nu>\nu_0, 
\end{equation}
or 
\begin{equation}\label{eqn condn Type B}
	\b_0 + h_0\g_\nu > \b^\prime + h^\prime\g_\nu \text{ for all } \nu>\nu_0. 
\end{equation} 
Now consider the expression 
\[  f(X) - f(z_\nu) = f_0(X) - f_0(z_\nu) + f^\prime(X) - f^\prime(z_\nu).  \]
The triangle inequality now yields that either 
\begin{equation}\label{eqn f Type A h_0}
	\begin{aligned}
		w\left( f(X) - f(z_\nu)  \right) &= \b_0 + h_0 \g_\nu = w\left( f_0(X) - f_0(z_\nu) \right) \\
		& < \b^\prime + h^\prime\g_\nu =  w\left(f^\prime(X)-f^\prime(z_\nu)\right),	
	\end{aligned}
\end{equation}
or, 
\begin{equation}\label{eqn f Type B hd}
	\begin{aligned}
		w\left( f(X) - f(z_\nu)  \right) & = \b^\prime + h^\prime\g_\nu  = w \left( f^\prime(X) - f^\prime(z_\nu)  \right)	\\
		&< \b_0 + h_0\g_\nu = w\left( f_0(X) - f_0(z_\nu) \right),
	\end{aligned}
\end{equation}
for all \( \nu>\nu_0 \). We will say that \(f\) is of \textbf{Type A} in the setup of (\ref{eqn f Type A h_0}), and \(f\) is of \textbf{Type B} in the situation of (\ref{eqn f Type B hd}). Therefore, 
\begin{equation}\label{eqn Type A beta(f) and h(f)}
	\begin{aligned}
		\b_K(X:f) &= \b_K(X:f_0),\\
		h_K(X:f) &= h_K(X:f_0),
	\end{aligned} \qquad \text{if \(f\) is of Type A.}
\end{equation}
On the other hand, 
\begin{equation}\label{eqn Type B beta(f) and h(f)}
	\begin{aligned}
		\b_K(X:f) &= \b_K(X:f^\prime) = \b_K(X:f-f_0),\\
		 h_K(X:f) &= h_K(X:f^\prime) = h_K(X:f-f_0),
	\end{aligned} \qquad \text{if \(f\) is of Type B.}
\end{equation}

\begin{Remark}\label{Rmk beta` and h`}
	We can present explicit descriptions of \(h^\prime\) and \(\b^\prime\) in the above scenario. By (\ref{eqn h(f) = j_w*(f)}) and Corollary \ref{Coro h(Q) = deg(Q) iff Q irr}, we have that
	\[ \#\Z(Q)\sect \Lim\C = d.  \]
	Hence every root of \(Q \) lies in \(\Lim\C\). It follows that \( v(z_\nu-a_i) = \g_\nu \) for every root \(a_i\) of \(Q\) and thus 
	\begin{equation}\label{eqn w_nu Q = d gamma_nu}
		w_\nu Q = vQ(z_\nu) = d\g_\nu \text{ for all } \nu<\l.
	\end{equation}
	Since \(  \deg f_i < \deg\C \), the values \( \{vf_i(z_\nu)\}_{\nu<\l} \) are ultimately stable. Denote this fixed value by \(\b_i\). Choose \( \nu_0<\l \) large enough such that the following statements hold true for all \( \nu>\nu_0 \):
	\begin{enumerate}[label=(\Roman*)]
		\item \( w_\nu f_i = v f_i(z_\nu) = \b_i \),
		\item there exists an integer \( h\in\{ 1,\dotsc , r \} \) satisfying
		\[  \b_h+ hd\g_\nu < \b_i + id\g_\nu \text{ for all } i\neq h, \, i\geq 1.   \]
	\end{enumerate}
	Condition (I) is possible by Lemma \ref{Lemma w_nu f = vf(z_nu)}, whereas the existence of \(h\) is guaranteed by \cite[Lemma 4]{Kaplansky1942}. From (\ref{eqn w_nu Q = d gamma_nu}) and Condition (I), we obtain
	\[  w_\nu (f_iQ^i) = \b_i + id\g_\nu = v\left((f_iQ^i)(z_\nu)\right)   \] 
	for all \( \nu>\nu_0 \). Hence Condition (II) and the triangle inequality yield
	\[ w_\nu f^\prime = v f^\prime(z_\nu) = \b_h + hd\g_\nu < \b_i + id\g_\nu  \]
	for all \( \nu>\nu_0 \) and \( i\neq h \). Consequently, 
	\begin{equation}\label{eqn w_nu (f-f_0) geq w(f-f_0)}
		w_\nu\left(  f^\prime(X) - f^\prime(z_\nu)  \right) \geq \b_h + hd\g_\nu \text{ for all } \nu>\nu_0.
	\end{equation}
	On the other hand, the fact that the sequence \( \{ vf^\prime(z_\nu) \} \) is ultimately strictly increasing implies that \( wf^\prime > vf^\prime(z_\nu) \) by \cite[Lemma 5.2]{KuhlmannVlahu2014}. As a consequence,
	\begin{equation}\label{eqn w((f-f_0) - (f-f_0)(z-nu))}
		w\left( f^\prime(X) - f^\prime(z_\nu) \right) = vf^\prime(z_\nu) = \b_h + hd\g_\nu \text{ for all } \nu>\nu_0.
	\end{equation}
	Since \( w_\nu\leq w \) on \( K[X] \), we conclude that
	\[  w_\nu \left( f^\prime(X) - f^\prime(z_\nu)  \right) = \b_h + hd\g_\nu = w\left( f^\prime(X) - f^\prime(z_\nu)  \right)  \]
	for all \( \nu>\nu_0\). Consequently, 
	\[ \b^\prime = \b_h \text{ and } h^\prime = hd.   \]
\end{Remark}

\pars We now present an analogue of Theorem \ref{Thm j-invariant characterization} in this setup:

\begin{Proposition}\label{Prop h(f) = j_{w_nu}(f) deg f > deg C}
	Assume that \( (K,v) \) is henselian. Let \( f\in K[X] \) with \( \deg f > \deg\C \) such that \( \Z(f)\sect\Lim\C = \emptyset \). Then the conclusions of Theorem \ref{Thm j-invariant characterization} hold true. 
\end{Proposition}

\begin{proof}
	We only consider the case when \(f\) is of Type B, since the complementary case is analogous. Fix \(\nu_0\) satisfying the conclusion of the preceding discussion, and take \( \nu>\nu_0 \). By Remark \ref{Rmk beta` and h`}, we have
	\begin{equation}\label{eqn key comparison Type A}
		\begin{aligned}
			w_\nu\left( f^\prime(X) - f^\prime(z_\nu) \right) &= w\left(  f^\prime(X) - f^\prime(z_\nu)   \right)\\
			& < w\left( f_0(X) - f_0 (z_\nu)  \right)\\
			& = w_\nu  \left( f_0(X) - f_0 (z_\nu)  \right),
		\end{aligned} 
	\end{equation}
	where the final equality follows exactly as in the proof of Proposition \ref{Prop h(f) = j_{w_nu} deg f < deg C}. Hence Proposition \ref{Prop in(f)=in(g) implies j(f)=j(g)} yields 
	\[  j_{w_\nu} \left( f(X) - f(z_\nu)  \right) = j_{w_\nu} \left( f^\prime(X) - f^\prime(z_\nu)  \right).  \]
	In light of Theorem \ref{Thm j-invariant characterization} and (\ref{eqn Type B beta(f) and h(f)}), we have
	\begin{equation}\label{eqn j(f-f(z))=h_0}
		j_{w_\nu} \left(  f(X) - f(z_\nu) \right) = h_K(X:f^\prime)=h^\prime = h_K(X:f).
	\end{equation}
	Thus we have proved the first assertion of Theorem \ref{Thm j-invariant characterization}. Now consider the Taylor expansions
	\begin{align*}
		f_0(X) - f_0(z_\nu) &= \sum_{i=1}^{\deg f_0} a_i (X-z_\nu)^i,\\
		f^\prime(X) - f^\prime(z_\nu) &= \sum_{i=1}^{n} b_i (X-z_\nu)^i.
	\end{align*}
	Then
	\[ f(X)-f(z_\nu) = \sum_{i=1}^{n} c_i (X-z_\nu)^i,  \]
	where 
	\[  c_i = \begin{cases}
		 a_i + b_i &\text{ when } i\leq \deg f_0,\\
		 b_i & \text{ when } i>\deg f_0.
	\end{cases}  \]
	In particular, we have
	\[ c_{h^\prime} = b_{h^\prime}.  \]
	As a consequence, we conclude from (\ref{eqn Type B beta(f) and h(f)}) and Theorem \ref{Thm j-invariant characterization} that
	\[ \b_K(X:f) = \b_K(X:f^\prime) = v b_{h^\prime} = vc_{h^\prime}.   \]
\end{proof}

\parm We combine Theorem \ref{Thm j-invariant characterization} and Propositions \ref{Prop h(f) = j_{w_nu} deg f < deg C} and \ref{Prop h(f) = j_{w_nu}(f) deg f > deg C} in the following result:

\begin{Theorem}\label{Thm central j-invariant characterization}
	Assume that \( (K,v) \) is henselian and let \( f \in K[X]\setminus K \) be arbitrary. Then the conclusions of Theorem \ref{Thm j-invariant characterization} hold true. 
\end{Theorem}

\begin{Proposition}\label{Prop h = p^e q}
	Assume that \( (K,v) \) is henselian and let \( f\in K[X] \) satisfy \( h_K(X:f) > 1 \). Then 
	\[ h_K(X:f) = p^e q,  \]
	where \( \gcd(p,q)=1 \) and \( e\in\ZZ_{\geq 1} \).
\end{Proposition}

\begin{proof}
	Since \( (K,v) \) is henselian, Corollary \ref{Coro h(Q) = deg(Q) iff Q irr} yields \( \deg Q = h_K(X:Q) \). Moreover, \cite[Proposition 7.4]{KuhlmannVlahu2014} shows that \( \deg Q \) is a \( p\)-power. Hence
	\[ h_K(X:Q) = p^m \text{ for some } m\in\ZZ_{\geq 1}.  \]
	If \( \Z(f) \) contains a limit of \(\C\), then the conclusion follows from (\ref{eqn h(Q)| h(f)}). Likewise, if \(f\) is of Type B, the assertion follows from (\ref{eqn Type B beta(f) and h(f)}) and Remark \ref{Rmk beta` and h`}. Finally, if \(f\) is of Type A, then \( h_K(X:f) = h_K(X:f_0) \) by (\ref{eqn Type A beta(f) and h(f)}). Since \( \deg f_0 < \deg\C \), \cite[Proposition 7.4]{KuhlmannVlahu2014} implies that \( h_K(X:f_0) \) is a power of \(p\), proving the proposition.
\end{proof}


\section{Proof of Theorem \ref{Thm K(X)^h:K(f)^h leq h}}\label{Section K(X)^h:K(f)^h leq h}

We will prove the following general result of which Theorem \ref{Thm K(X)^h:K(f)^h leq h} is an immediate consequence in light of Theorem \ref{Thm central j-invariant characterization}.  

\begin{Proposition}
	Assume that \((K,v)\) is a henselian valued field and take \(f\in K[X]\setminus K\). Then
	\[  [ K(X)^h : K(f(X))^h ] \le j_{w_\nu} \left( f(X) - f(z_\nu) \right) \text{ for all } \nu<\l.  \]
\end{Proposition}

The following argument follows the general strategy of \cite[Theorem 9.1]{KuhlmannVlahu2014}.

\begin{proof}
	Take any \(\nu<\l\) and expand 
	\[ f(X) - f(z_\nu) = \sum_{i=1}^{n} c_i(X-z_\nu)^i.   \]
	Set \( h:= j_{w_\nu} \left( f(X) - f(z_\nu) \right) \). Corollary \ref{Coro j(f) when deg(Q)=1} yields that
	\[  w_\nu \left( f(X) - f(z_\nu)  \right) = w_\nu \left( c_h (X-z_\nu)^h \right) \leq w_\nu \left( c_i(X-z_\nu)^i   \right)   \]  
	for all \( i \neq h \); moreover, the inequality is strict for \( i>h \). Since 
	\[ w_\nu \left( c_i(X-z_\nu)^i \right) = vc_i + i\g_\nu = w \left( c_i(X-z_\nu)^i \right) \text{ for all } i,  \] 
	we have
	\[ w\left( c_h (X-z_\nu)^h \right) \leq w\left( c_i(X-z_\nu)^i \right) \text{ for all } i\neq h,  \]
	and 
	\[ w\left( f(X) - f(z_\nu) \right) \geq w_\nu \left( f(X) - f(z_\nu) \right) = w\left( c_h (X-z_\nu)^h \right).  \]
	Since \( (K(X)|K,w) \) is an immediate extension, we can take \( b,d\in K \) such that
	\[  vd = \g_\nu \text{ and } v( bc_hd^h)=0.  \]
	Set
	\[ \tilde{X}:= \dfrac{X-z_\nu}{d}.  \]
	The preceding inequalities therefore imply
	\begin{equation}\label{eqn [K(X)^h : K(f)^h] leq h}
		\begin{aligned}
			w \bigl( b(f(X) - f(z_\nu) )\bigr) &\geq w \left( bc_hd^h \tilde{X}^h \right) = 0 \\
			&  \leq w \left( bc_i d^i \tilde{X}^i   \right)  \text{ for all } i\neq h.
		\end{aligned}
	\end{equation}		
	Consider the polynomial 
	\[ g(T):= \sum_{i=1}^{n} bc_id^i T^i - b\left( f(X) - f(z_\nu)  \right) \in \mathcal{O}_{K(f)}[T].  \]
	Since the second inequality in (\ref{eqn [K(X)^h : K(f)^h] leq h}) is strict for \( i>h \), we obtain that
	\[  g(T)w = \sum_{i=0}^{h} bc_id^iv T^i - b\left( f(X)-f(z_\nu) \right)w \]
	is a polynomial of degree \(h\). Using the Strong Hensel's Lemma, that is, property 3) of \cite[Theorem 4.1.3]{EnglerPrestel2005}, we have a decomposition 
	\[ g(T) = g_1 (T) g_2 (T)  \]
	over \( K(f)^h \) with 
	\[ \begin{aligned}
		\deg g_1(T) = \deg g_1(T) w &= \deg g(T)w = h,\\
		g_2(T)w &= 1 . 
	\end{aligned}   \]
	Since \( g_2(T)w=1 \), the polynomial \( g_2(T) \) cannot admit a root of value \(0\). In particular, \( \tilde{X} \) cannot be a root of \( g_2 (T)\). On the other hand, \( \tilde{X} \) is a root of \( g(T) \). We have thus arrived at
	\[ g_1(\tilde{X}) = 0.  \]
	On observing that \( K(X) = K(\tilde{X}) \), we conclude
	\[ [K(X)^h:K(f)^h] \leq \deg g_1(T) = h,  \]
	thereby completing the proof. 
	\end{proof}

\begin{Corollary}\label{Coro h(f)=1}
	Let \( (K,v) \) be a henselian valued field and \( f(X)\in K[X] \). Assume that
	\[ j_{w_\nu} \left( f(X) - f(z_\nu) \right) = 1 \]
	for some \( \nu < \l \). Then
	\[ K(X)^h = K(f(X))^h.   \]
\end{Corollary}

\section{Artin-Schreier reduction of relative approximation degree over perfect fields}\label{Section A-S reduction}

We begin with a few observations about the Artin-Schreier operator \( \P \). Define
\[ \P(K[X]) := \left\{ g^p - g \mid g\in K[X]  \right\}.  \]
Since the map \( g\longmapsto g^p - g \) is additive in characteristic \(p\), the set \( \P(K[X]) \) forms an additive subgroup of \( K[X] \). In particular, if
\[  f_i \equiv g_i \bmod \P(K[X])  \]
for \( i=1,2 \), then
\[ f_1 + f_2 \equiv \left( g_1 + g_2 \right) \bmod\P (K[X]).  \]

\begin{Proposition}\label{Prop h(X:F) leq p}
	Assume that \((K,v)\) is henselian, perfect and \( f\in K[X]\setminus K \). Then there exists \(F\in K[X]\) such that \(F\equiv f \bmod \P(K[X])\) such that
	\[ h_K(X:F) \le p. \]
\end{Proposition}

\begin{proof}
	The assertion is trivial when \( h_K(X:f) = 1 \). So we assume that \( h:= h_K(X:f) > 1 \) and set \( \b:= \b_K(X:f) \). Then 
	\[ w\left( f(X)-f(z_\nu) \right) = \b + h\g_\nu  \]
	for all \( \nu \) sufficiently large. Since \( \{\g_\nu\} \) is strictly increasing and \( h>0 \), the sequence \( \{ \b+h\g_\nu \} \) is strictly increasing as well. Hence exactly one of the following holds: either
	\begin{equation}\label{eqn beta + h gamma > 0}
		\b+ h\g_\nu > 0 \text{ for all } \nu \text{ sufficiently large}, 
	\end{equation}
	or
	\begin{equation}\label{eqn beta + h gamma < 0}
		\b+ h\g_\nu < 0 \text{ for all } \nu.
	\end{equation}
	Choose \( \nu_0 \) sufficiently large so that the conclusions of Theorem \ref{Thm central j-invariant characterization} apply, and fix \( \nu>\nu_0 \) accordingly. Consider the expansion
	\[ f(X) = f(z_\nu) + \sum_{i=1}^{n} c_i (X-z_\nu)^i.  \]
	Take \(d\in K\) such that
	\[ v(c_1 + d) +\g_\nu < vc_h + h\g_\nu,  \]
	and set 
	\[ F(X) := f(X) - \P(d(X-z_\nu) ).  \]
	Then
	\[  F(X) = f(z_\nu) + (c_1 + d)(X-z_\nu) + (c_p - d^p) (X-z_\nu)^p + \sum_{i\neq 1,p} c_i (X-z_\nu)^i.  \]
	The choice of such a \(d\) is straightforward in the setup of (\ref{eqn beta + h gamma < 0}). If (\ref{eqn beta + h gamma > 0}) holds, then since \((K,v)\) is perfect, the value group \(vK\) is \(p\)-divisible, and hence does not have a smallest positive element. Thus the choice of \(d\) is again validated. In either case, observe that
	\[ v(c_1+d)+\g_\nu < vc_h + h\g_\nu \le vc_1 + \g_\nu, \]
	whereby \( v(c_1+d) = vd \) by the triangle inequality. In the scenario of (\ref{eqn beta + h gamma > 0}), we further have
	\[ 0 < v(c_1 + d) + \g_\nu = vd+ \g_\nu \Longrightarrow v(c_1 + d) + \g_\nu < vd^p + p\g_\nu.  \]
	As a consequence, 
	\[  v(c_1 + d) + \g_\nu < \min\{ vc_p, vd^p \}+ p\g_\nu \leq v(c_p - d^p) + p\g_\nu.  \]
	It follows that
	\[  v(c_1+d) + \g_\nu < \min \left\{ v(c_p - d^p) + p\g_\nu, \, vc_h + h\g_\nu \right\} \le \min \left\{ v(c_p - d^p) + p\g_\nu, \, vc_i + i\g_\nu \mid i\ne 1,p \right\}, \]
	whereby 
	\[ j_{w_\nu} \left( F(X) - F(z_\nu) \right) = 1  \]
	by Corollary \ref{Coro j(f) when deg(Q)=1}. If (\ref{eqn beta + h gamma < 0}) holds, our choice of \(d\) ensures
	\[ p(vd + \g_\nu) < vd + \g_\nu = v(c_1 + d) + \g_\nu < vc_p + p\g_\nu,    \]
	whereby \( vd^p < vc_p \). Hence \( v(c_p - d^p) =  vd^p \) by the triangle inequality. It follows that
	\begin{align*}
		v\left( c_p - d^p \right) + p\g_\nu &= p(vd + \g_\nu) = p(v(c_1 + d) + \g_\nu)\\
		&  < v (c_1 + d) + \g_\nu < vc_h + h\g_\nu \le \min \left\{ vc_i + i\g_\nu \mid i\ne 1,p \right\}.
	\end{align*}
	Therefore Corollary \ref{Coro j(f) when deg(Q)=1} yields that
	\begin{equation*}
		j_{w_\nu} \left( F(X) - F(z_\nu) \right) = p.
	\end{equation*}
	
	\pars We have thus obtained
	\[ 1 \le j_{w_\nu} \left( F(X) - F(z_\nu)\right) \le p  \]
	in either case. By Theorem \ref{Thm central j-invariant characterization}, we can choose \( \mu > \nu \) such that 
	\[ h_K(X:F) = j_{w_\mu} \left( F(X) - F(z_\mu) \right).  \] 
	Write 
	\[ F(X) - F(z_\mu) = F(X) - F(z_\nu) + F(z_\nu) - F(z_\mu). \]
	Thus the higher-order monomials in \( X-z_\nu \) remain unchanged when we move from \( F(X) - F(z_\nu) \) to \( F(X) - F(z_\mu) \). Since \( (z_\mu, \g_\nu) \) is also a minimal pair of definition for \( w_\nu \) by Proposition \ref{Prop pair of defn}, we infer from Corollary \ref{Coro j(f) when deg(Q)=1} that 
	\[  j_{w_\nu} \left( F(X) - F(z_\mu) \right) \leq j_{w_\nu} \left( F(X) - F(z_\nu) \right) \le p. \]
	Moreover, \( j_{w_\nu} \left( F(X) - F(z_\mu) \right) \) cannot be zero since \( z_\mu \) is a root of \( F(X) - F(z_\mu) \) and \( v (z_\mu - z_\nu) = \g_\nu \). Consequently,
	\[ j_{w_\nu} \left( F(X) - F(z_\mu) \right) \ge 1. \]
	Finally, \( z_\mu \) being a common center for both \(w_\mu\) and \(w_\nu\), the fact \( \g_\mu > \g_\nu \) implies that \( j_{w_\mu} (-) \leq j_{w_\nu} (-) \). Since the relative approximation degree is a positive integer, we conclude that 
	\[ 1 \le h_K(X:F) = j_{w_\mu} \left( F(X) - F(z_\mu) \right) \le p.  \] 
\end{proof}

\begin{Remark}
	It follows from Proposition \ref{Prop h = p^e q} that
	\[  h_K(X:F) = 1 \text{ or } p. \]
	Moreover, \( h_K(X:F) = 1 \) whenever (\ref{eqn beta + h gamma > 0}) holds. However, the situation of (\ref{eqn beta + h gamma < 0}) poses the essential difficulty to achieving henselian rationality over perfect fields, as we will illustrate in the next section.
\end{Remark}

\section{An example}\label{Sect example}

Let \( (k,v) \) denote the field \(\FF_p(t)\) equipped with the \(t\)-adic valuation, where \(p\) is an odd prime. Fix an extension of \(v\) to \(\overline{k}\) and set 
\[  K := \left( k^h \right)^{1/p^\infty}.  \]
For \(n\in\ZZ_{>0}\), define 
\[ z_n := \sum_{k=1}^{n} t^{-1/p^k} \in K. \]
Then \(\C:= \{z_n\}\) is a pseudo-convergent sequence in \((K,v)\) with
\[ \g_n := v(z_n - z_{n+1}) = -\dfrac{1}{p^{n+1}}.  \]
Define 
\[  a := \sum_{k=1}^{\infty} t^{-1/p^k}. \]
Then \( a^p - a = \dfrac{1}{t} \) and hence \(a\) is algebraic over \(K\). Observe that
\[ v (a-z_n) = \g_n \text{ for all } n > 0.  \]
As a consequence, \( a \in \Lim\C \). Moreover, it has been observed in \cite[Example 3.12]{Kuhlmann2011Defect} that
\[ v(a-K) = \left( vK \right)_{<0}. \]
Since \( \{\g_n\} \) is cofinal in \(\left( vK \right)_{<0}\), this implies that \(\C\) has no limits in \(K\). Therefore, 
\[ \text{\(\C\) is a pseudo-convergent sequence in \((K,v)\) of algebraic type.}  \]
Finally, combining the Lemma of Ostrowski together with \cite[Theorem 3]{Kaplansky1942} yields that the degree of the associated minimal polynomial to an algebraic type pseudo-convergent sequence over a henselian valued field is a \(p\)-power. Consequently, we conclude that
\[ Q(X) := X^p - X - \dfrac{1}{t} \text{ is an associated minimal polynomial to } \C.   \]

\begin{Proposition}
	There exists an immediate extension \( (K(X)|K,w) \) such that \(X\in\Lim\C\).
\end{Proposition}

\begin{proof}
	Take a formal power series \( \eta := \displaystyle\sum_{k=1}^{\infty} b_{i_k} t^{i_k} \in \FF_p((t)) \) which is not algebraic over \(\FF_p(t)\), where \(b_{i_k}\neq 0\) for all \(k\). Extend the sequence \(\C\) to the sequence
	\[ \mathcal{A} := \left\{ a_1, a_2, \cdots , a_\o, a_{\o+1}, a_{\o+2}, \cdots   \right\} \subset \overline{K},   \]
	where
	\begin{align*}
		a_\nu := \begin{cases}
			z_\nu &\quad \nu < \o,\\
			a &\quad \nu = \o,\\
			a + \displaystyle\sum_{k=1}^{n} b_{i_k} t^{i_k} &\quad \nu = \o + n.
		\end{cases}
	\end{align*}
	Then \(\mathcal{A}\) is a pseudo-convergent sequence by construction, with the corresponding successive differences 
	\begin{align*}
		\d_\nu := v(a_\nu - a_{\nu+1}) = 
		\begin{cases}
			\g_\nu &\quad \nu < \o,\\
			i_{n+1} &\quad \nu = \o + n. 
		\end{cases}
	\end{align*}
	Since the sequence \( \{\d_\nu\} \) approaches \(\infty\), \(\mathcal{A}\) is a Cauchy sequence and thus has a unique limit in the completion \(\widehat{\overline{K}}\). It is well-known that given an algebraic extension \(F'|F\) of valued fields, their completions satisfy \( \widehat{F} \subseteq \widehat{F'} \). Consequently, 
	\[ \FF_p((t)) = \widehat{k} \subseteq \widehat{\overline{k}} = \widehat{\overline{K}}.  \]
	From our construction of \(\mathcal{A}\), it follows that
	\[ a + \eta \in \Lim\mathcal{A}. \]
	Since \( \eta \in \FF_p((t)) \), we have that \(a+\eta\)  is the unique limit of \(\mathcal{A}\) in \(\widehat{\overline{K}}\). Moreover, \(\eta\) is not algebraic over \(\FF_p(t)\). Thus \(a+\eta \notin \overline{K}\). Therefore,
	\[ \mathcal{A} \text{ is of transcendental type in } (\overline{K},v). \] 
	 Applying \cite[Theorem 2]{Kaplansky1942}, we obtain an immediate extension \( (\overline{K}(X)|\overline{K},w) \) such that \(X\in\Lim\mathcal{A}\). Hence in particular, 
	 \[ X \in \Lim\C. \]
	 It remains to show that the extension \( (K(X)|K,w) \) is immediate. Take \( b\in\overline{K} \) such that \( w\left( X-b \right) = \d_\nu \). If \(\nu<\o\), then since \(a_\nu = z_\nu \in K\), we obtain
	 \[ \deg_K(b) \geq \deg_K(a_\nu). \]	 
	 If \( \nu\geq\o \), then \(b\) is also a limit of \(\C\). Since \(Q\) is the associated minimal polynomial to \(\C\), this again implies
	 \[ \deg_K(b) \geq \deg Q = \deg_K(a) = \deg_K(a_\nu).  \]
	 \cite[Theorem 5.1]{AlexandruPopescuZaharescu1990b} then implies
	 \[ wK(X) = \Union_{\nu} vK(a_\nu) \quad \text{and} \quad K(X)w = \Union_{\nu}K(a_\nu)v.   \]
	 Observe that \( K(a_\nu) = K \) for all \(\nu<\o\) and \( K(a_\nu) = K(a) \) otherwise. Furthermore, the extension \( (K(a)|K,v) \) is immediate, as observed in \cite[Theorem 3.12]{Kuhlmann2011Defect}. As a consequence, 
	 \[ wK(X) = \Union_{\nu} vK(a) = vK \quad \text{and} \quad K(X)w = \Union_{\nu}K(a)v = Kv.   \]
\end{proof}

\begin{Remark}
	The existence of such a transcendental formal power series \(\eta\) is validated by the fact that \( \trdeg [\FF_p((t)) : \FF_p(t)] = \infty \). Numerous such examples can be constructed using the language of non-automatic sequences and employing Christol's theorem. For a precise example, consider the sequence 
	\begin{align*}
		b_n := \begin{cases}
			1 &\quad n \text{ is a prime},\\
			0 &\quad \text{otherwise}.
		\end{cases}
	\end{align*}
	Then \(\{b_n\}_{n\geq 1}\) is not \(p\)-automatic by \cite[Example 25]{ALLOUCHE20221}, and hence 
	\[  \sum_{n=1}^{\infty} b_n t^n = \sum_{q \text{ prime}} t^q  \in \FF_p((t)) \]
	is not algebraic over \(\FF_p(t)\) by Christol's theorem \cite{BSMF_1980__108__401_0}. 
\end{Remark}

We will now work with the polynomial 
\begin{equation}\label{eqn f}
	f(X) := \left( t^{-1/p} - t^{-1} \right) X + bX^2 + X^{p+1}, \text{ where } b^p = -2^{1-p}.
\end{equation}
Then for any \( n\geq 1 \), we have
\[ f(X) - f(z_n) = \sum_{i = 1}^{p+1} c_i \left( X - z_n \right)^i,  \]
where
\begin{equation}\label{eqn c_i}
	c_i=
	\begin{cases}
		z_n^p + 2bz_n + t^{-1/p} - t^{-1} & \quad i=1,\\
		b &  \quad i=2,\\
		z_n & \quad i=p,\\
		1 & \quad i=p+1,\\
		0 & \quad \text{otherwise}.
	\end{cases}
\end{equation}
A direct computation yields
\[  c_1^p = Q(z_n)^p - Q(z_n) - z_n.  \]
Now,
\[ Q(z_n) = z_n^p - z_n - \dfrac{1}{t} = -t^{-1/p^n} \Longrightarrow vQ(z_n) = -\dfrac{1}{p^n}.  \] 
Moreover, \( vz_n = -1/p \). On applying the triangle inequality to the above expression, we then obtain
\[ vc_1 = -\dfrac{1}{p^2} \text{ whenever } n\geq 3.  \]
For each \(n\), define 
\[ w_n := v_{z_n, \g_n}.  \]
For any \(n\geq 3\), we then obtain
\[  w_n \left( c_p (X-z_n)^p \right) < w_n \left( c_1 (X-z_n) \right) < w_n \left( c_{p+1} (X-z_n)^{p+1} \right) < w_n \left( c_2 (X-z_n)^2  \right).  \]
Hence 
\[   j_{w_n} \left( f(X)-f(z_n) \right) = p   \]
by Corollary \ref{Coro j(f) when deg(Q)=1}. Theorem \ref{Thm central j-invariant characterization} then yields the following:

\begin{Proposition}
	\( h_K(X:f) = p. \)
\end{Proposition}

Our primary goal for this section is to prove the following result:

\begin{Theorem}
	Take a polynomial \( g(X) \in K[X] \). Then
	\[  h_K \left( X : f - g^p + g \right) \ge p. \]
\end{Theorem}

\begin{proof}
	We prove by contradiction. Set \( F:= f - g^p + g \) and suppose that \(h_K(X:F) < p\). Then 
	\[   h_K(X:F) = 1  \]
	by Proposition \ref{Prop h = p^e q}. By Theorem \ref{Thm central j-invariant characterization}, 
	\[ j_{w_n} \left( F(X) - F(z_n) \right) = 1 \text{ for all } n \text{ sufficiently large}.  \]
	Choose such an \(n\) and write 
	\[ g(X) - g(z_n) = \sum_{i=1}^{t} d_i \left( X - z_n \right)^i.  \]
	Setting \( d_i = 0 \) if required, we further assume that \( t \ge p+1 \). Using (\ref{eqn c_i}), we obtain the Taylor series expansion 
	\begin{equation}\label{eqn F-expansion}
		\begin{aligned}
			F(X)-F(z_n)
			={}&(c_1+d_1)(X-z_n)
			+(b+d_2)(X-z_n)^2 \\
			&+(z_n+d_p-d_1^p)(X-z_n)^p
			+(1+d_{p+1})(X-z_n)^{p+1} \\
			&+\sum_{\substack{3\le i\le t\\ i\neq p,p+1}}
			d_i(X-z_n)^i
			-\sum_{i=2}^{t} d_i^p(X-z_n)^{ip}.
		\end{aligned}
	\end{equation}
We will now show that
\begin{equation}\label{eqn v(c_1+d_1) < p gamma_n}
	v \left( c_1 + d_1 \right) < p\g_n.
\end{equation}
Since \( j_{w_n} (F(X) - F(z_n)) = 1 \), applying Corollary \ref{Coro j(f) when deg(Q)=1} to (\ref{eqn F-expansion}) yields
\[ v(c_1 + d_1) + \g_n < v \left( 1 + d_{p+1} \right) + (p+1)\g_n.  \]
Therefore if \( vd_{p+1} \neq 0 \), the triangle inequality implies (\ref{eqn v(c_1+d_1) < p gamma_n}). Now assume that \( v d_{p+1} = 0 \). Choose \(r\in\NN\) such that  
\begin{align*}
	v d_{p^i (p+1)} &= 0 \text{ for all } 0\le i \le r,\\
	v d_{p^{r+1} (p+1)} & \ne 0. 
\end{align*}
Note that here we use the convention \( d_i = 0 \) for \( i > t \). Since \( vd_{p^r(p+1)} = 0 \), we have \( p^r(p+1) \le t \). Comparing the linear term and the monomial of degree \( p^{r+1}(p+1) \) in (\ref{eqn F-expansion}), and applying Corollary \ref{Coro j(f) when deg(Q)=1}, we obtain
\begin{align*}
	v (c_1 + d_1) + \g_n &< v \left( d_{p^{r+1}(p+1)} - d_{p^r(p+1)}^p \right) + p^{r+1}(p+1)\g_n \\
	& \le v d_{p^r(p+1)}^p + p^{r+1}(p+1)\g_n = p^{r+1}(p+1)\g_n\\
	&\le (p+1)\g_n,
\end{align*}
thereby again yielding (\ref{eqn v(c_1+d_1) < p gamma_n}). On the other hand, observe that 
\[ c_1^p + z_n = Q(z_n)^p - Q(z_n).  \]
Since \( vQ(z_n) = -1/p^n = p\g_n \), the above expression gives us 
\[ v\left( c_1^p + z_n \right) = p^2\g_n. \]
Thus (\ref{eqn v(c_1+d_1) < p gamma_n}) can be modified as
\begin{equation}\label{eqn pv(c_1+d_1) < v(c_1^p + z_n)}
	pv(c_1+d_1) < v\left( c_1^p + z_n \right).
\end{equation} 
Corollary \ref{Coro j(f) when deg(Q)=1} and (\ref{eqn F-expansion}) further yield that
\[ v(c_1+d_1)+\g_n < v(z_n + d_p - d_1^p) + p\g_n.  \]
Since \( v(c_1+d_1)+\g_n < (p+1)\g_n < 0 \), we have \( pv(c_1+d_1) + p\g_n < v(c_1+d_1) + \g_n \). Therefore, we obtain 
\[ pv(c_1+d_1) < v(z_n + d_p - d_1^p)  \]
from the above expression. Employing the triangle inequality, we have
\[ pv(c_1+d_1) = v(c_1^p + z_n + d_p).  \]
Together with (\ref{eqn pv(c_1+d_1) < v(c_1^p + z_n)}), this implies
\[  pv(c_1+d_1) = vd_p < v(c_1^p + z_n).  \]
In particular, 
\[ d_p \ne 0. \]
For any \(i>1\), a comparison of the linear term and the monomial of degree \(p^i\) in (\ref{eqn F-expansion}) in combination with an application of Corollary \ref{Coro j(f) when deg(Q)=1} yields 
\[ v(c_1 + d_1) + \g_n < v \left( d_{p^i} - d_{p^{i-1}}^p \right) + p^i \g_n,  \]
whereby 
\[  p^i v(c_1 + d_1) < v \left( d_{p^i} - d_{p^{i-1}}^p \right).  \] 
Since \( pv(c_1 + d_1) = v d_p \), this implies that
\begin{equation}\label{eqn vd_p condition}
	p^{i-1} vd_p < v \left( d_{p^i} - d_{p^{i-1}}^p \right) \text{ for all } i > 1.
\end{equation}
Take \(s\in\NN\) such that 
\begin{align*}
	d_{p^i} &\ne 0 \text{ for all } 1 \le i \le s,\\
	d_{p^{s+1}} &= 0,
\end{align*}
where we use the convention \(d_i = 0\) for \(i > t\). Repeated application of the triangle inequality in (\ref{eqn vd_p condition}) yields that
\[ v d_{p^s} = p^{s-1} vd_p. \]
On the other hand, plugging in \( i = s+1 \) in (\ref{eqn vd_p condition}), we obtain
\[ p^s vd_p < v\left( d_{p^{s+1}} - d_{p^s}^p \right) = p vd_{p^s} = p^s vd_p, \]
thereby reaching a contradiction. We thus have the theorem. 
\end{proof}

The above example demonstrates that, unlike the tame case, one cannot expect to reduce the relative approximation degree to one merely by replacing an Artin-Schreier polynomial with an equivalent representative. Consequently, the reduction step underlying Kuhlmann's proof of henselian rationality over tame fields cannot be carried over directly to perfect fields. On the other hand, the obstruction does not imply a failure of henselian rationality. In fact, the associated Artin-Schreier function field is henselian rational, as illustrated underneath.  

\begin{Proposition}
	Fix an extension of \(w\) to \(\overline{K(X)}\). The extension \( (K(X,\th)|K,w) \) is henselian rational, where \( \th^p - \th = f \).
\end{Proposition}

\begin{proof}
	Since the pseudo-convergent sequence \(\mathcal{A}\) is of transcendental type, we obtain that \(wQ = vQ\left(a+ b_{i_1}t^{i_1}+ \dotsc + b_{i_n}t^{i_n}\right)\) for sufficiently large \(n\). Observe that
	\begin{align*}
		Q\left(a+ b_{i_1}t^{i_1}+ \dotsc + b_{i_n}t^{i_n}\right) &= \left(a+ b_{i_1}t^{i_1}+ \dotsc + b_{i_n}t^{i_n}\right)^p - \left(a+ b_{i_1}t^{i_1}+ \dotsc + b_{i_n}t^{i_n}\right) - \dfrac{1}{t} \\
		&= \left(b_{i_1}t^{i_1}+ \dotsc + b_{i_n}t^{i_n}\right)^p - \left(b_{i_1}t^{i_1}+ \dotsc + b_{i_n}t^{i_n}\right).
	\end{align*}
It follows from the triangle inequality that
\[ wQ = i_1 \ge 1.  \]
Choose \(c\in K\) with \( -1 < vc < -\dfrac{1}{p^3} \) and set 
\[ g := f + cQ.  \]
Observe that
\[ g(X) - g(z_n) = (c_1 - c)(X-z_n) + b(X-z_n)^2 + (z_n + c)(X-z_n)^p + (X-z_n)^{p+1}.   \]
Fix some \(n\ge 4\). Reducing modulo the Artin-Schreier operator, we have \( g(X) \equiv G(X) \bmod \P(K[X]) \), where
\begin{equation}\label{eqn G(X)-G(z_n)}
	G(X)  = g(z_n) + \left( c_1 - c + z_n^{1/p} + c^{1/p}  \right)(X-z_n) + b(X-z_n)^2 + (X-z_n)^{p+1}.
\end{equation}
Using the expression of \(c_1\) in (\ref{eqn c_i}), a straightforward computation yields
\[  c_1 - c + z_n^{1/p} + c^{1/p} = (Q(z_n) - c) - (Q(z_n) - c)^{1/p}.  \]
Since \(n\ge 4\), we have \( vc < -1/p^3 < -1/p^n = vQ(z_n) \). Hence we obtain from the triangle inequality that
\[  v \left( c_1 - c + z_n^{1/p} + c^{1/p}  \right) = vc.  \]
Moreover,
\[ vc+ \g_n < (p+1)\g_n < 2\g_n.  \]
As a consequence, (\ref{eqn G(X)-G(z_n)}) and Corollary \ref{Coro h(f)=1} yield
\[  h_K(X:G) = j_{w_n} (G(X) - G(z_n)) = 1. \]
Finally, 
\[ w(g-f) = vc + wQ = vc+ i_1 \ge vc + 1 > 0. \]
Therefore, we obtain from Hensel's Lemma that
\[  g \equiv f \bmod \P(K(X)^h).  \]
It follows that
\[ G\equiv f \bmod \P(K(X)^h).  \]
Hence there exists another Artin-Schreier generator \(\a\) over \(K(X)^h\) such that
\[ K(X)^h(\th) = K(X)^h(\a) \text{ and } \a^p - \a = G.  \]
Since \(h_K(X:G) = 1\), we conclude that
\[ K(X)^h(\th) = K(X)^h(\a) = K(G(X))^h(\a) = K(\a)^h.  \]
Thus \( (K(X,\th)|K,w) \) is henselian rational.
\end{proof}


\section{Sufficient conditions for henselian rationality over rank one perfect fields}\label{Section henselian rationality over IC} Throughout this section, we assume that \(K\) is a perfect valued field of rank one and characteristic \(p>0\) and \( (F|K,w) \) is an immediate extension of valued function fields of transcendence degree one. Fix an extension of \(w\) to \(\overline{F}\), denoted again by \(w\). For any \(X\in F\setminus K\) transcendental over \(K\), there is a pseudo-convergent sequence \(\C\subset (K,v)\) without any limit in \(K\) and with \(X\) as a limit. Such a sequence will be called a pseudo-convergent sequence \textit{corresponding} to the extension \((K(X)|K,w)\).

\subsection{Finite descent of henselian rationality}

\begin{Lemma}\label{Lemma henselian ratioanlity pull down finite extns}
	Let \((L|K,v)\) be a finite subextension of \( (F^h|K,w) \) and assume that \( (F^h|L,w) \) is henselian rational. Then \( (F|K,w) \) is also henselian rational.
\end{Lemma}

\begin{proof}
	Write \(F^h = L(X)^h\) for some \(X\in F^h\setminus L\) transcendental over \(L\). Since \(K\) is perfect, the Primitive Element Theorem asserts that \(L = K(a)\) for some \(a\in L\). Consider the Krasner constant of \(a\):
	\[  \o_K(a) := \left\{ v (\s a - a) \,\middle|\, \s\in \Gal(\overline{K}|K) \text{ and } \s a \neq a \right\}.  \]
	Choose \(c\in K^\times\) such that \( w(cX) > \o_K(a) \) and set \( Y:= cX + a\). Then \( w(Y-a) > \o_K(a)\). A variant of Krasner's Lemma \cite[Lemma 2.21]{Kuhlmann2004BadPlaces} then yields that
	\[  a \in K(Y)^h. \]
	Consequently, 
	\[ L(X)^h = K\left( a, \dfrac{Y-a}{c} \right)^h  = K(a,Y)^h = K(Y)^h. \]
	Since \( F^h = L(X)^h \), we conclude that \((F|K,w)\) is henselian rational.
\end{proof}


\subsection{Henselian rationality over implicit constant fields}

The above lemma shows that it can be beneficial to investigate when \(F^h\) is henselian rational over the relative algebraic closure of \(K\) in \(F^h\). This leads us to the theory of \textit{implicit constant fields}, formulated by Kuhlmann in \cite{Kuhlmann2004BadPlaces}.

\begin{Definition}
	Let \( (\O|k,\nu) \) be any extension of valued fields and fix an extension of \(\nu\) to \(\overline{\O}\). The \textbf{implicit constant field} of the extension is defined as the relative algebraic closure of \(k\) in the henselization of \(\O\), that is,  
	\[ IC(\O|k,\nu) := \overline{k} \sect \O^h.  \]
\end{Definition}

We want to investigate when \(F^h\) is henselian rational over \(IC(F|K,w)\). In the rank one setup, the desired conclusion holds after passage to the absolute ramification field of \((K,v)\).  
\begin{Definition}
	The \textbf{absolute ramification group} of a perfect valued field \( (K,v) \) is defined as
	\[  G^r := \left\{ \s\in \Gal(\overline{K}|K) \,\middle|\, v\left( \s a - a \right) > va \text{ for all } a\in\overline{K} \setminus \{0\}  \right\}. \]
	The fixed field of \(G^r\) is called the \textbf{absolute ramification field} of \( (K,v) \), and is denoted by \(K^r\). 
\end{Definition}
Then
\[  vK^r = \left\{ \a\in \QQ\tensor_{\ZZ} vK \mid n\a\in vK \text{ for some } n \text{ coprime to }p  \right\},  \]
and 
\[ K^rv = \text{separable-algebraic closure of } Kv,  \]
by Part c) of \cite[Theorem 5.7]{Kuhlmann2000LocalUniformization}. \( (K,v) \) being perfect implies that the value group \( vK \) is \( p\)-divisible and the residue field \(Kv\) is perfect. Consequently,
\[ vK^r \text{ is a divisible group and } K^rv \text{ is algebraically closed}.  \]
Hence when \( K = K^r \), we are in the setting where immediate extensions are controlled by Artin-Schreier theory.

\begin{Proposition}\label{Prop pcs tr type over IC}
	Assume that \(K = K^r\). Take some \(X \in F\setminus K\) transcendental over \(K\) and set \(I := IC(K(X)|K,w)\). Take a pseudo-convergent sequence \(\C \subset (I,v)\) corresponding to the immediate extension \((I(X)|I,w)\). Then \(\C\) is of transcendental type. 
\end{Proposition}

\begin{proof}
	Suppose that \(\C\) is of algebraic type. Take an associated minimal polynomial \(Q\in I[X]\) and a root \(a\in \Z(Q)\). The perfectness of \(K\) implies that \(I\) is also perfect. Moreover, \(I^r = I\cdot K^r = I\). Thus \(vI\) is divisible and \(Iv\) is algebraically closed. It now follows from \cite[Lemma 5.1]{Kuhlmann2019EliminationII} that \(I(a)|I\) is a tower of Artin-Schreier extensions. We thus have an Artin-Schreier generator \(\th\) with 
	\[ I\subsetneq I(\th) \subseteq I(a).  \]
	Observe that \((I(\th)|I,v)\) is immediate, and hence an Artin-Schreier defect extension. Since \(I\) is perfect, this extension is an \textit{independent} defect extension, in the language of \cite{Kuhlmann2010ArtinSchreier}. As a consequence, since \((I,v)\) is a rank one valued field, we conclude from \cite[Lemma 2.14(k)]{Kuhlmann2010ArtinSchreier} that
	\begin{equation}\label{eqn v(theta-I) = 0^-}
		v\left( \th - I \right) = (vI)_{<0}.
	\end{equation}
	Write \(\th = f(a)\) where \(f(X)\in I[X]\) with \(\deg f < \deg Q\). Let \(\C = \left\{z_\nu\right\}_{\nu<\l}\) and \(\g_\nu := v(z_\nu - z_{\nu+1})\). Then
	\[  w\left( f(X) - f(z_\nu) \right) = \b + h\g_\nu \text{ for all } \nu<\l,  \]
	where \( \b:= \b_I(X:f) \) and \(h:= h_I(X:f)\). Consider the factorization \(f(X)-f(z_\nu) = \prod (X-c_i)\) over \(\overline{I}\). Since none of the \(c_i\)'s are limits of \(\C\) and \(a \in \Lim\C\), we obtain that \( w(X-a) > v(X-c_i) \) for all \(i\). Therefore,
	\[ v\left( f(a) - f(z_\nu) \right) = w\left( f(X) - f(z_\nu) \right) = \b + h\g_\nu. \]
	The triangle inequality then yields that 
	\begin{equation}\label{eqn w(f-theta) > beta + h gamma_nu}
		w\left( f(X) - f(a) \right) > \b+h\g_\nu \text{ for all } \nu<\l.
	\end{equation} 
	Observe that
	\[  \b+h\g_\nu = v(\theta - f(z_\nu)) \in v(\th - I). \]
	If \( \left\{ \b + h\g_\nu \right\} \) is not cofinal in \(v(\th-I)\), then there exists some \(z\in I\) such that \( v(\th-z) > \b+ h\g_\nu \) for all \(\nu\). As a consequence,
	\[  v(f(z_\nu) - z) = \b + h\g_\nu \text{ for all } \nu<\l. \]
	In other words, setting \(f':= f(X) - z \in I[X]\), we have
	\[ vf'(z_\nu) = \b + h\g_\nu \text{ for all } \nu<\l. \]
	Thus \(\{ vf'(z_\nu) \}\) is strictly increasing along \(\C\). However, this contradicts the minimality of \(Q\), since \(\deg Q > \deg f = \deg f'\). We therefore obtain that
	\[ \left\{ \b + h\g_\nu \right\} \text{ is cofinal in } v(\th-I).   \]  
	From (\ref{eqn v(theta-I) = 0^-}) and (\ref{eqn w(f-theta) > beta + h gamma_nu}), we conclude that 
	\[  w(f(X)-\th) \ge 0. \]
	Now \(\th\) being an Artin-Schreier generator yields that the Krasner constant \(\o_I(\th) = 0\). If \( w(f(X) - \th) > 0 \), then \cite[Lemma 2.21]{Kuhlmann2004BadPlaces} implies that
	\[ \th\in I(f(X))^h \subseteq I(X)^h.  \]
	On the other hand, \(I\) is relatively algebraically closed in \(K(X)^h\) by definition. Since \(K(X)^h = I(X)^h\) and \(\th\notin I\), the above containment relation is not possible. Thus
	\[  w(f(X) - \th) = 0. \]
	Since \( (I(\th,X)|I,w) \) is immediate, we can take some \(c\in K\) such that \( (f(X)-\th)w = cv \). Consequently, \( w(f(X)-\th-c) > w(f(X)-\th) = 0 \). Observe that \(\th+c\) is also an Artin-Schreier generator of \(I(\th)|I\). The preceding arguments then again lead us to a similar contradiction. Therefore, it follows that the pseudo-convergent sequence \(\C\) must be of transcendental type.
\end{proof}

In the presence of pseudo-convergent sequences of transcendental type, we can employ Kuhlmann's theory of normal forms. This leads us to the next observation:

\begin{Proposition}\label{Prop henselian rationality over IC}
	Assume that \(K = K^r\). Take some separating transcendental element \(X \in F\setminus K\) and set \(I := IC(K(X)|K,w)\). Then,
	\[ (F^h|I,w) \text{ is henselian rational}.  \]
\end{Proposition}

\begin{proof}
	Observe that \( wK(X) \) is divisible and \(K(X)w\) is algebraically closed. Employing \cite[Lemma 5.1]{Kuhlmann2019EliminationII}, the extension \( F^h|K(X)^h \) can be decomposed as a tower of Artin-Schreier extensions: 
	\[ K\subset K(X)^h \overset{\text{A-S}}{\subset} K(X)^h(\th_1) \overset{\text{A-S}}{\subset} \cdots \overset{\text{A-S}}{\subset} K(X)^h(\th_1, \dotsc , \th_n) = F^h,  \]
	where \(\th_i\) is an Artin-Schreier generator over \( K(X)^h (\th_1, \dotsc, \th_{i-1}) \). The inclusions \( K\subseteq I \subset K(X)^h \) yield that
	\[ K(X)^h = I(X)^h.  \]
	Thus the above tower can be rewritten as
	\[ K\subseteq I \subset I(X)^h \overset{\text{A-S}}{\subset} I(X)^h(\th_1) \overset{\text{A-S}}{\subset} \cdots \overset{\text{A-S}}{\subset} I(X)^h(\th_1, \dotsc , \th_n) = F^h.  \]
	Since \((I,v)\) is a rank one valued field, we can assume that \( f(X):= \th_1^p - \th_1 \in I[X] \) by \cite[Lemma 4.1]{Kuhlmann2019EliminationII}. From Proposition \ref{Prop pcs tr type over IC}, we obtain that any pseudo-convergent sequence corresponding to the extension \( (I(X)|I,w) \) is of transcendental type. Consequently, in view of Lemmas 4.2 and 4.7 of \cite{Kuhlmann2019EliminationII}, we can assume that \( I(X)^h = I(f(X))^h \). It follows that
	\[ I(X)^h(\th_1) = I(f(X))^h(\th_1) = I(\th_1)^h.  \]
	Therefore, we have the modified tower
	\[ K\subseteq I \subset I(\th_1)^h \overset{\text{A-S}}{\subset} I(\th_1)^h (\th_2) \overset{\text{A-S}}{\subset} \cdots \overset{\text{A-S}}{\subset} I(\th_1)^h(\th_2, \dotsc , \th_n) = F^h.   \]   
	Repeated applications of the preceding arguments yield the tower
	\[ K\subseteq I \subseteq I_1 \subseteq \cdots \subseteq I_{n-1} \subset I_{n-1}(\th_n)^h = F^h, \  \]
	where 
	\[ I_{j+1} := IC (I_j(\th_{j+1}) | I_j, w).  \]
	Thus, 
	\[ (F^h | I_{n-1},w) \text{ is henselian rational}.  \]
	Now, \(I\) is relatively algebraically closed in \(K(X)^h\) by definition. Hence the algebraic extension \(I_1|I\) is linearly disjoint to \(K(X)^h|I\). Moreover, \(I_1\subset I(\th_1)^h = K(X)^h(\th_1)\). Thus the compositum \( I_1 \cdot K(X)^h \subseteq K(X)^h(\th_1) \). It follows that 
	\[ [I_1:I] = [I_1\cdot K(X)^h : K(X)^h] \le [K(X)^h(\th_1) : K(X)^h] = p.  \]
	Repeated iterations of this argument yield that \( [I_{n-1}:I] < \infty \). Since \( (F^h|I_{n-1},w) \) is henselian rational, the assertion now follows from Lemma \ref{Lemma henselian ratioanlity pull down finite extns}.
\end{proof}

Therefore, henselian rationality is obtained if we can exhibit the existence of some separating transcendental element \(X\in F\setminus K\) such that \(IC(K(X)|K,w)\) is a finite extension of \(K\). We observe, however, that the finiteness of the implicit constant field is a property of the extension \((F|K,w)\) itself, and is thus independent of the choice of \(X\).

\begin{Proposition}
	Let notations and assumptions be as in Proposition \ref{Prop henselian rationality over IC}. Set \( \widetilde{I} := IC(F|K,w) \). Then
	\[ \widetilde{I}|I \text{ is a finite extension}. \]
\end{Proposition}

\begin{proof}
	Since \(I\) is relatively algebraically closed in \(K(X)^h\), the algebraic extension \(\widetilde{I}|I\) is linearly disjoint to \(K(X)^h|I\). It follows that
	\[ [\widetilde{I}:I]  = [\widetilde{I}\cdot K(X)^h : K(X)^h] \le [F^h : K(X)^h] < \infty. \]
\end{proof}


\subsection{Proof of Theorem \ref{Thm henselian rationality over IC}}

\begin{proof}
	Take a separating transcendental element \(X\in F\setminus K\) and set \( I:= IC(K(X)|K,w) \). Then \( (F^h|I,w) \) is henselian rational by Proposition \ref{Prop henselian rationality over IC}. Since \( I\subseteq IC(F|K,w) = L \), it follows that \( (F^h|L,w) \) is henselian rational. The final assertion is now a direct consequence of Lemma \ref{Lemma henselian ratioanlity pull down finite extns}. 
\end{proof}


\subsection{Type II extensions and finiteness of the implicit constant field}

We now provide a class of immediate function field extensions \((F|K,w)\) of rank one valued fields where \(IC(F|K,w)\) is a finite extension of \(K\), and hence \((F|K,w)\) is henselian rational. Take some transcendental element \(X\in F\setminus K\). Fix a completion \((\widehat{K},\widehat{v})\) of \((K,v)\). It has been observed in \cite{Dutta2023MathNach} that there exists a common extension of \(w\) and \(\widehat{v}\) to \(\overline{\widehat{K}}(X)\), which we will denote by \(\widehat{w}\). The extension \( \left( \widehat{K}(X)|\widehat{K}, \widehat{w} \right) \) is said to be \textit{induced} by \((K(X)|K,w)\). Depending on the behaviour of the induced extensions, the extensions \( (K(X)|K,w) \) are classified in two groups: Type I and Type II. It has been observed in \cite{Dutta2023MathNach} that \( (K(X)|K,w) \) is of Type II if it satisfies the following equivalent criteria: 
\sn (C1) the extension \( \left( \widehat{K}(X)|\widehat{K}, \widehat{w} \right) \) is not immediate, 
\nn (C2) the extension \( \left( \widehat{K}(X)|\widehat{K}, \widehat{w} \right) \) is valuation transcendental,
\nn (C3) there is a Cauchy sequence corresponding to the extension \( \left( \overline{K}(X) | \overline{K}, w \right) \) (thus the example constructed in Section \ref{Sect example} is of Type II).\\

\begin{Proposition}\label{Prop Type II}
	Let notations and assumptions be as in Proposition \ref{Prop henselian rationality over IC}. Further, assume that the extension \( (K(X)|K,w) \) is of Type II. Then,
	\[ (F|K,w) \text{ is henselian rational}. \]
\end{Proposition}

\begin{proof}
	 Fix a completion \((\widehat{K},\widehat{v})\) of \((K,v)\) and consider the induced extension \( \left( \widehat{K}(X) | \widehat{K}, \widehat{w} \right) \). Set \( \widehat{I} := IC (\widehat{K}(X)|\widehat{K},\widehat{w}) \). Since \( (K,v) \) is a perfect valued field of rank one, it is relatively algebraically closed in \(\widehat{K}\). Thus \(I|K\) and \(\widehat{K}|K\) are linearly disjoint. Moreover, it has been observed in \cite[Equation 9.1]{Dutta2023MathNach} that 
	\[  I \cdot \widehat{K} \subseteq \widehat{I} .  \]
	Consequently, 
	\[ [I:K] = [I \cdot \widehat{K} : \widehat{K}] \le [\widehat{I} : \widehat{K}].  \]
	Since \( (K(X)|K,w) \) is of Type II, using the characterization (C2) above and employing \cite[Theorem 1.1]{Dutta2021}, there exists an element \(a\) algebraic over \(\widehat{K}\) such that \( \widehat{I} \subseteq \widehat{K}(a) \). As a consequence, 
	\[ [I:K] \le [\widehat{K}(a) : \widehat{K}] < \infty.  \]
	The assertion now follows from Proposition \ref{Prop henselian rationality over IC} and Lemma \ref{Lemma henselian ratioanlity pull down finite extns}.
\end{proof}


\section{Concluding remarks and open problems}\label{Section open problems} A consequence of Theorem \ref{Thm henselian rationality over IC} and Proposition \ref{Prop Type II} is the following result:

\begin{Corollary}
	Let \((K,v)\) be a perfect valued field of positive characteristic satisfying \(K=K^r\), and let \((F|K,w)\) be an immediate extension of valued fields of transcendence degree one. If there exists a separating transcendental element \(X\in F\setminus K\) of Type II, then \((F|K,w)\) is henselian rational. 
\end{Corollary}

Since every such extension is either of Type I or Type II, the preceding corollary leaves precisely the following case open: 
\begin{Problem}\label{Problem}
	Let \((K,v)\) be a perfect valued field of positive characteristic satisfying \(K=K^r\), and let \((F|K,w)\) be an immediate extension of valued fields of transcendence degree one. Set
	\[ L= IC(F|K,w). \]
	Assume that \( (K(X)|K,w) \) is of Type I for every separating transcendental element \(X\in F\setminus K\). Is \( (F|K,w) \) henselian rational? Equivalently, given that
	\[ F^h = L(Y)^h \]
	for some element \(Y\), does there necessarily exist \(Z\in F^h\), transcendental over \(K\), such that
	\[ F^h = K(Z)^h? \]
\end{Problem}

An affirmative answer to Problem \ref{Problem} would settle the henselian rationality problem over rank one perfect fields coinciding with their absolute ramification fields. This extends the algebraically-closed base field setup of Step \ref{Step B} in Kuhlmann's program. We mention here that this is a strict generalization, since there exist non-algebraically closed henselian valued perfect fields \( (K,v) \) of rank one satisfying \( K = K^r \). An example is furnished underneath.

\begin{Example}\label{Eg K not alg closed}
	Consider \(\FF_p(t)\) equipped with the \(t\)-adic valuation \(v\). Fix an extension of \( v\) to \(\overline{\FF_p(t)}\), denoted again by \(v\). Let \( (k,v) \) denote the henselization of \( (\FF_p(t), v) \).  Set
	\[ k_1:= k^{1/p^\infty} \text{ and } K:= k_1^r,  \]
	where \(k^{1/p^\infty}\) denotes the perfect closure of \(k\). We will show that \(K\) is not algebraically closed by exhibiting a non-trivial Artin-Schreier extension \(K(a)|K\). Take \( a\in\overline{k} \) such that
	\[ a^p - a = 1/t.  \]
	Since \( va = -1/p \notin \ZZ = vk \), we obtain that \( k(a)|k \) is a non-trivial Artin-Schreier extension. In particular, it is separable. Thus it is linearly disjoint to \(k_1\), the perfect closure of \(k\). Consequently,
	\begin{equation*}
		[k_1(a):k_1] = p.
	\end{equation*}
	Moreover, the Fundamental Inequality yields
	\[  vk(a) = \dfrac{1}{p}\ZZ \text{ and } k(a)v = \FF_p.  \]
	Since \(k_1 = k^{1/p^\infty}\), the field \(k_1(a)\) is precisely the perfect closure of \(k(a)\). Thus \( vk_1(a) \) is the \(p\)-divisible closure of \(vk(a)\) and \(k_1(a)v\) is the perfect closure of \(k(a)v\). It follows that
	\[ vk_1(a) = \dfrac{1}{p^\infty}\ZZ = vk_1 \text{ and } k_1(a)v = \FF_p = k_1v.  \]
	Thus \( (k_1(a)|k_1,v) \) is an immediate extension. Since \(K = k_1^r\), the extension \( (K|k_1,v) \) is \textit{defectless}, and hence we obtain from \cite[Lemma 2.5]{Kuhlmann2010ArtinSchreier} that
	\[ [K(a):K] = p.  \]
	Hence \( K(a)|K \) is a non-trivial Artin-Schreier extension, and in particular, \(K\) is not algebraically closed. 
\end{Example}

The preceding discussion shows that a successful resolution of the problem of henselian rationality boils down to two remaining descent problems. The first is descent from \(IC(F|K,w)\) to \(K\) in the Type I case under the assumption \(K= K^r\). The second is descent from \(K^r\) to an arbitrary henselian perfect base field \(K\), an analogue of Step \ref{Step C} in Kuhlmann's program. 


\section*{Declaration of generative AI and AI-assisted technologies in the manuscript preparation process} During the preparation of this manuscript, the authors used OpenAI ChatGPT as an auxiliary tool for mathematical discussion, as well as for improving the organization, clarity and language of the manuscript. After using this service, the authors reviewed and edited the content as needed, and take full responsibility for the content of the published article.


\bibliographystyle{alpha}
\bibliography{references}

\end{document}